\documentclass[letterpaper,10pt]{article}
\usepackage[english]{babel}

\usepackage[letterpaper,top=2cm,bottom=2cm,left=3cm,right=3cm,marginparwidth=1.75cm]{geometry}

\usepackage{comment}
\usepackage{amsmath}
\usepackage{amsthm}
\usepackage{amssymb}
\usepackage{graphicx}
\usepackage[all]{xy}
\usepackage[colorlinks=true, allcolors=blue]{hyperref}

\theoremstyle{plain}
\newtheorem{theorem}{Theorem}[section] 
\newtheorem{proposition}[theorem]{Proposition} 
\newtheorem{corollary}[theorem]{Corollary} 
\newtheorem{conjecture}[theorem]{Conjecture} 
 
\newtheorem{lemma}[theorem]{Lemma} 
 
\theoremstyle{definition}
\newtheorem{example}[theorem]{Example} 
\newtheorem{remark}[theorem]{Remark}

\newcommand{\hyp}{\operatorname{hyp}}
\newcommand{\nonhyp}{\operatorname{non-hyp}}
\newcommand{\QQ}{\mathcal Q}
\newcommand{\HH}{\mathcal H}
\newcommand{\TQ}{\widetilde{\mathcal Q}}

\newlength{\halfbls}
\title{Filtration and splitting of the Hodge bundle on the non-varying strata of quadratic differentials}

\author{Dawei Chen\thanks{Research of D.C. is supported by the National Science Foundation Grant DMS-2001040 and Simons Travel Support for Mathematicians 635235.}, Fei Yu\thanks{Research of F.Y. is supported by the National Natural Science Foundation of China under Grant No. 11871422 and the Fundamental Research Funds for the Central
 Universities 107101*17221022301(2021FZZX001-01)}}

\begin{document}
\maketitle
\begin{abstract}
We describe the Harder--Narasimhan filtration of the Hodge bundle for Teichm\"uller curves in the non-varying strata of quadratic differentials appearing in \cite{CM2}. Moreover, we show that the Hodge bundle on the non-varying strata away from the irregular components can split as a direct sum of line bundles. As applications, we determine all individual Lyapunov exponents of algebraically primitive Teichm\"uller curves in the non-varying strata and derive new results regarding the asymptotic behavior of Lyapunov exponents.
\end{abstract}

\setcounter{tocdepth}{1}
\tableofcontents

\section{Introduction}

Based on a limited number of computer experiments carried out approximately twenty years ago, Kontsevich and Zorich made an observation that the sum of Lyapunov exponents remains constant for all Teichm\"uller curves in a stratum of Abelian differentials when the sum of the genus and the number of zeros is less than seven. This conjecture was first proved in \cite{CM1} by showing that Teichm\"uller curves in those low genus strata are disjoint from divisors of Brill--Norther type in moduli spaces of curves. Another proof was given in \cite{YZ1} by using filtration of the Hodge bundle. For quadratic differentials, the non-varying property of the hyperelliptic strata was established in \cite[Section 2.3]{EKZ}. Many non-hyperelliptic strata of quadratic differentials in low genus were also shown to be non-varying in \cite{CM2}. These non-varying strata are closely related to affine geometry \cite{Ch1, Ch2}. As for strata with varying sums of Lyapunov exponents, various results on bounding those varying sums can be found in \cite{CM1, CM2, YZ2, Fo}.

The Harder--Narasimhan filtration of the Hodge bundle of Teichm\"uller curves have been constructed in the non-varying strata of Abelian differentials as well as in the hyperelliptic loci of all Abelian strata \cite{YZ1}. The (normalized) Harder--Narasimhan polygon of the filtration can be described by using associated Weierstrass exponents $w_i$ \cite[Definition 3.7]{YZ1}. Interesting applications of the filtration have also appeared such as in \cite{Au, BHM}. 

Given a stratum of quadratic differentials, via the canonical double cover it can lift into the corresponding stratum of Abelian differentials, where the Hodge bundle over the lifted image consists of the invariant part and the anti-invariant part with respect to the involution of the canonical double cover. In this paper, we study the splitting of the Hodge bundle on the non-varying strata of quadratic differentials listed in \cite{CM2}. Moreover, we determine the Harder--Narasimhan filtration of the invariant and anti-invariant parts of the Hodge bundle of Teichm\"uller curves in these non-varying strata and evaluate the associated Weierstrass exponents $w_i^+, w_i^-$. Since different connected components of a stratum can have distinct non-varying properties, when speaking of a non-varying stratum we mean a non-varying component of the stratum. Our main result is stated as follows.

\begin{theorem}\label{main}
\begin{itemize}
\item[(1)] 
Let $\QQ_g(d_1,\ldots, d_n)$ be a non-varying stratum of quadratic differentials listed in \cite{CM2} that are not hyperelliptic or irregular. Then the Hodge bundle on $\QQ_g(d_1,\ldots, d_n)$ is a direct sum of line bundles. Moreover, for any Teichm\"uller curve $C$ in these strata the Weierstrass exponents of the Harder--Narasimhan filtration of the Hodge bundle on $C$ are given by $w_i^+$ for $i=1,\ldots,g$ being the $(g-i+1)$-st smallest number in the set 
$$\left\{\frac{2k}{d_j+2}\mid 0< 2k\leq d_j+1,\ j=1,\ldots,n\right\},$$
$w_1=w_1^-=1$, and  $w_i^-$ for $i=2,\ldots,g_{\rm eff}$ being the $(g_{\rm eff}-i+1)$-st smallest number in the set 
$$\left\{1-\frac{2k}{d_j+2}\mid 0<2k\leq d_j+1,\ j=1,\ldots,n\right\}.$$
\item[(2)] The values of $w_i^+$ and $w_i^-$ for the Hodge bundle on Teichm\"uller curves in the non-varying irregular strata in \cite{CM2} are given by Table \ref{irr}. 
\item[(3)] The Hodge bundle splits as a direct sum of line bundles on the hyperelliptic strata of quadratic differentials.  
Moreover, the values of $w_i^+$ and $w_i^-$ for the Hodge bundle on Teichm\"uller curves in these hyperelliptic strata are given by Theorem \ref{thm:hyp-strata}.
\end{itemize}
\end{theorem}

As explained in \cite[Section 6]{YZ2} and \cite[Section 4]{BHM}, for algebraically primitive Teichm\"uller curves, the individual Lyapunov exponents and Weierstrass exponents coincide with each other (and in general those Lyapunov exponents in the part of real multiplication come from Weierstrass exponents). We thus conclude the following result. 

\begin{corollary}
The individual Lyapunov exponents of any algebraically primitive Teichm\"uller curve in the non-varying strata of quadratic differentials are given by the corresponding Weierstrass exponents in Theorem~\ref{main}. 
\end{corollary}

In the course of the proof of Theorem~\ref{main} we will actually show the splitting 
of the Hodge bundle on a partial compactification of the strata by adding boundary points of type similar to the cusps of Teichm\"uller curves (see Section~\ref{sec:quad} for more details). Likewise all of our results on the Harder--Narasimhan filtration stated for Teichm\"uller curves hold for any compact subvariety contained in the partial compactification. 

As an application of our method we will also show in Theorem~\ref{thm:abelian-splitting} that the Hodge bundle splits as a direct sum of line bundles on any non-varying stratum of Abelian differentials in \cite{CM1} that is not a spin component of even parity. 

Another application of our results is about the asymptotic behavior of Lyapunov exponents. Kontsevich and Zorich \cite{KZ} conjectured that the large genus limit of the second largest Lyapunov exponent satisfies that 
$$\lim_{g\rightarrow \infty}\lambda_2=1$$
for the hyperelliptic strata of Abelian differentials $\mathcal{H}_g(2g-2)^{\hyp}$ and $\mathcal{H}_g(g-1,g-1)^{\hyp}$, while for all other strata they conjectured that 
$$\lim_{g\to\infty}\lambda_2=\frac{1}{2}.$$
The hyperelliptic case of the conjecture has been settled by using the relation between the Harder--Narasimhan polygon and the Lyapunov polygon \cite{Yu, EKMZ}. Here we can generalize this case as follows, where $\mathcal{Q}_0(n,-1^{n+4})$ below corresponds to the hyperelliptic strata  $\mathcal{H}_g(2g-2)^{\rm hyp}$ (for odd $n$) and $\mathcal{H}_g(g-1,g-1)^{\rm hyp}$ (for even $n$) via the canonical double cover. 

\begin{corollary}
For the strata of quadratic differentials $\mathcal{Q}_0(d_1,\ldots,d_k)$ (with $n=\max(d_1,\cdots,d_k)$), $\mathcal{Q}_1(n,-1^n)$, and $\mathcal{Q}_1(n,1,-1^{n+1})$, we have 
$$\lim_{n\to\infty}\lambda_m^-=1$$
for any fixed $m$. Moreover for the hyperelliptic strata $\mathcal{Q}(2(g-k)-3,2(g-k)-3,2k+1,2k+1)^{\rm hyp}$, $\mathcal{Q}(2(g-k)-3,2(g-k)-3,4k+2)^{\rm hyp}$, and $\mathcal{Q}(4(g-k)-6,4k+2)^{\rm hyp}$, we have
$$\lim_{k\to\infty}\lambda_m^+=\lim_{g\to\infty}\lambda_m^+=\lim_{k\to\infty}\lambda_m^-=\lim_{g\to\infty}\lambda_m^-=1.$$
\end{corollary}

We will also investigate upper bounds of $w_i^+$ and $w_i^-$ for all Teichm\"uller curves in Section~\ref{sec:bounds}. Additionally, we will show in Corollary \ref{no} that all Lyapunov exponents of any Teichm\"uller curve in those non-varying strata of quadratic differentials are nonzero.

We end the introduction by looking into some future directions. We expect that many results stated for Teichm\"uller curves and for the partial compactification of the strata should still hold for higher-dimensional affine invariant subvarieties and for a complete compactification of the strata. Moreover, due to the algebraic nature of filtration of the Hodge bundle, we hope that our results and techniques can shed light on the problem of computing sums of Lyapunov exponents via intersection numbers (see \cite[Conjecture 4.3]{CMS}). Finally in a general context beyond Teichm\"uller dynamics, one can work with the Hodge bundle on the strata of moduli spaces of pointed curves where the stratification is induced by the hierarchy of Weierstrass semigroups at the marked points. We plan to study these questions in future work. 

\section*{Acknowledgements} We thank Martin M\"oller and Anton Zorich for related discussions over the years. We also thank Giovanni Forni, Shuai Guo, Zhe Sun, Wanyuan Xu and Kang Zuo for helpful communications.  

\section{Filtration and splitting of direct images} 
\label{sec:filtration}
In this section we will establish some general results about filtration and splitting of direct image vector bundles on families of curves. Later on we will apply these results to study the Harder--Narasimhan filtration and splitting of the Hodge bundle on strata of quadratic differentials. 

We first recall the definition of the Harder--Narasimhan filtration (see e.g., \cite[Section 1.3]{HL} for more details). Let $B$ be an $n$-dimensional compact variety with a polarization given by an ample divisor class $A$. For a vector bundle $V$ of rank $r$ on $B$, define the degree and slope of $V$ respectively as 
$$\deg V = c_1(V)\cdot A^{n-1}\quad {\rm and}\quad \mu (V) = \frac{\deg V}{r}.$$
A Harder--Narasimhan filtration ${\rm HN}(V)$ for $V$ is an increasing filtration 
$$0 = {\rm HN}_0(V) \subset {\rm HN}_1(V) \subset \cdots \subset {\rm HN}_k(V) = V$$ 
such that the graded quotients 
$${\rm gr}_i^{{\rm HN}(V)} = {\rm HN}_i(V) / {\rm HN}_{i-1}(V)  $$ 
are semistable vector bundles for $i=1, \ldots, k$ and that 
$$ \mu({\rm gr}_1^{{\rm HN}(V)}) > \mu({\rm gr}_2^{{\rm HN}(V)}) > \cdots > \mu({\rm gr}_k^{{\rm HN}(V)}). $$
The Harder--Narasimhan filtration of $V$ exists uniquely with respect to the polarization $A$. Moreover if $B$ is an algebraic curve, then the Harder--Narasimhan filtration exists uniquely regardless of the polarization. 

Let $r_i$ be the rank of ${\rm gr}_i^{{\rm HN}(V)}$ where $r = r_1+\cdots +r_k$. Consider the sequence 
\begin{eqnarray}
\label{eq:mu_i}
\mu_1\geq \mu_2 \geq \cdots \geq \mu_r 
\end{eqnarray}
where each $\mu({\rm gr}_i^{{\rm HN}(V)})$ appears $r_i$ times in the decreasing order. The Harder--Narasimhan polygon of $V$ is defined as the convex hull in $\mathbb R^2$ spanned by the points 
$$ (0,0), (1,\mu_1), (2,\mu_2), \ldots, (r, \mu_1+\cdots +\mu_r). $$
Later on we will use a normalized version of the Harder--Narasimhan polygon of the Hodge bundle on a Teichm\"uller curve $C$, where the normalization mods out one-half of the orbifold Euler characteristic of $C$ (see e.g., \cite[Section 4.1]{CM2}). 

Define the direct sum of graded quotients of the Harder--Narasimhan filtration of $V$ as 
$${\rm grad}({\rm HN}(V)) = \bigoplus_{i=1}^k {\rm gr}_i^{{\rm HN}(V)}. $$
If $V = V_1 \oplus \cdots \oplus V_m$ is a direct sum of vector bundles, then we have 
$${\rm grad}({\rm HN}(V)) = {\rm grad}({\rm HN}(V_1)) \oplus \cdots \oplus {\rm grad}({\rm HN}(V_m))$$
and every $\mu({\rm gr}_i^{{\rm HN}(V_j)})$ equals $\mu({\rm gr}_l^{{\rm HN}(V)})$ for some $l$. 

In what follows we focus on the case of direct image vector bundles on a family of curves.  

\begin{lemma}\label{sp1} 
Let $f\colon S\to B$ be a flat family of curves with a section $E$. Then for $a\geq 1$ and any divisor class $M$ in $S$, we have  
$$f_*\mathcal{O}_{aE}(M)  
=\bigoplus^{a-1}_{i=0} f_*\mathcal{O}_{E}(M-iE). $$
Moreover, suppose $B$ is compact $n$-dimensional with an ample divisor class $A$ such that $E^2 \cdot (f^{*}A)^{n-1} < 0$ in $S$. Then the Harder--Narasimhan filtration of $f_*\mathcal{O}_{aE}(M)$ with respect to $A$ is $$0\subset f_*\mathcal{O}_{E}(M-(a-1)E)\subset\cdots\subset f_*\mathcal{O}_{(a-1)E}(M-E)\subset f_*\mathcal{O}_{aE}(M).$$
\end{lemma}

\begin{proof}
The scheme structure of $aE$ is characterized by 
$$ 0\to \mathcal O_S(-aE) \to \mathcal O_S \to \mathcal O_{aE} \to 0.$$
Note that $f$ induces a retraction from the subscheme $aE$ to $B$, where the restriction of $f$ to the reduced support $ E\to B$ is an isomorphism. As in \cite[Lemma A.3]{CM2} and \cite[Proposition 1.1]{BE} we have the exact sequence 
$$ 0\to f_{*}\mathcal O_{(a-1)E}(M-E) \to f_{*}\mathcal O_{aE}(M) \to f_{*}\mathcal O_{E}(M)\to 0 $$ 
which indeed splits because the identity element of $f_{*}\mathcal O_{aE}$ lifts the identity element of $f_{*}\mathcal O_{E}$ as induced by the retraction. Therefore, 
$$f_{*}\mathcal O_{aE}(M) = f_{*}\mathcal O_{E}(M)\oplus f_{*}\mathcal O_{(a-1)E}(M-E) $$
and the first claim thus follows by induction. 

For the other claim, since $E^2\cdot (f^{*}A)^{n-1}<0$, we have 
$$(M-(a-i)E)\cdot E\cdot (f^{*}A)^{n-1}<(M-(a-i+1)E)\cdot E\cdot (f^{*}A)^{n-1}.$$ 
We thus obtain the desired filtration with the property that the degrees of the graded quotient bundles are strictly decreasing. The claim now follows from the uniqueness of the Harder--Narasimhan filtration. 
\end{proof}

\begin{lemma}\label{sp2} Let $M$ be a divisor class in $S$ and $a_i \geq 0$. If $h^0(M|_F)$ and $h^0((M-\sum a_iD_i)|_F)$ are constant for every fiber $F$, and if moreover $h^0(M|_F)=h^0\big((M-\sum a_iD_i)|_F\big)+\sum a_i$, then 
$$f_*\mathcal{O}_S(M)/f_*\mathcal{O}_S\left(M-\sum a_iD_i\right)=\bigoplus f_*\mathcal{O}_{a_iD_i}(M) = \bigoplus_i \bigoplus_{j=0}^{a_i-1} f_{*}\mathcal O_{D_i}(M - jD_i).$$
\end{lemma}
\begin{proof} Since $D_i$ and $D_j$ are disjoint for $i\neq j$, we have the exact sequence
$$0\rightarrow \mathcal{O}_S\left(M-\sum a_iD_i\right)\rightarrow \mathcal{O}_S(M)\rightarrow\bigoplus \mathcal{O}_{a_iD_i}(M)\rightarrow 0.$$
We then obtain the long exact sequence
$$0\rightarrow f_*\mathcal{O}_S\left(M-\sum a_iD_i\right)\rightarrow f_*\mathcal{O}_S(M)\rightarrow \bigoplus f_{*} \mathcal{O}_{a_iD_i}(M)\stackrel{\delta}{\rightarrow} $$
$$ \stackrel{\delta}{\rightarrow} R^1f_*\mathcal{O}_S\left(M-\sum a_iD_i\right)\rightarrow R^1f_*\mathcal{O}_S(M)\rightarrow 0.$$
By assumption all terms in the sequence are vector bundles, and moreover by Riemann--Roch
$$\text{rank\,} R^1f_*\mathcal{O}_S\left(M-\sum a_iD_i\right)=\text{rank\,} R^1f_*\mathcal{O}_S(M).$$
We thus conclude that 
$$0\rightarrow f_*\mathcal{O}_S\left(M-\sum a_iD_i\right)\rightarrow f_*\mathcal{O}_S(M)\rightarrow \bigoplus f_{*}\mathcal{O}_{ a_iD_i}(M)\rightarrow 0 $$
is exact as desired. The last expression of splitting follows from Lemma~\ref{sp1}. 
\end{proof}

\section{Quadratic differentials and Teichm\"uller curves}
\label{sec:quad}
In this section we will review and establish some fundamental results about quadratic differentials and Teichm\"uller curves. We will use $\QQ$ to denote a projectivized stratum of quadratic differentials, i.e., modulo the $\mathbb C^{*}$-scaling of differentials. 
Here we only consider primitive quadratic differentials that are not global squares of Abelian differentials. We will also use a partial compactification $\TQ$ by adding the horizontal boundary divisor in the sense of \cite{BCGGM1, BCGGM2, BCGGM3}, which parameterizes stable quadratic differentials that do not vanish on any irreducible component of the underlying nodal curve. Note that $\TQ$ contains all cusps of Teichm\"uller curves in the stratum \cite[Corollary 4.5]{CM2}, hence this partial compactification indeed completely compactifies every Teichm\"uller curve. From now on when speaking of a Teichm\"uller curve we always mean its closure in $\TQ$.  

For $(Y,q)$ in the stratum of quadratic differentials $\QQ(d_1,\ldots,d_m)$, consider the canonical double cover $(X,\omega)$ of $(Y, q)$, which comes with an involution $\tau$ on $X$ such that $\sigma\colon X\to Y \cong X/\tau$, $\tau^{*}\omega = -\omega$, and $\sigma^{*}q = \omega^2$, where $\omega$ is an Abelian differential on $X$.
Since by assumption $q$ is primitive, the double cover $X$ remains connected. The cohomology of $X$ can be split into the $\tau$-invariant and $\tau$-anti-invariant parts. Let $g=g(Y)$ denote the genus of $Y$, and define 
$$g_{\rm eff}=g(X)-g.$$ 

Consider the following map  
$$\phi\colon \mathcal{Q}_g(\ldots,d_i,\ldots,d_j,\ldots)\rightarrow \mathcal{H}_{g+g_{\rm eff}}\Big(\ldots,\frac{1}{2}d_i,\frac{1}{2}d_i,\ldots,d_j+1,\ldots\Big)$$
induced by the canonical double cover construction which is 
branched at singularities of odd order $d_j$. As in \cite{BCGGM2} the map $\phi$ extends and lifts $\TQ$ to the corresponding partial compactification of the stratum of Abelian differentials. We thus obtain two universal families, denoted by $f'$ and $f$ respectively, in the following commutative diagram: 
$$\xymatrix{
    S  \ar[dr]_{f}
                &  &  \ar[ll]^{\sigma}  S' \ar[dl]^{f'}    \\
                & \TQ                }
$$

If $d_j$ is odd, then we define $D'_j$ to be the section of the fibration $f'\colon S'\rightarrow \TQ$ over the section $D_j$ of the $j$-th singularity of $q$. For odd $d_j$ we have 
$$\sigma_*D_j'=D_j, \quad \sigma^*D_j=2D_j'. $$
If $d_j$ is even, we denote the two distinct sections over $D_j$ as $D'_{j,1}$ and $D'_{j,2}$. For even $d_j$ we have
$$\sigma_*(D'_{j,1}+D'_{j,2})=2D_j, \quad \sigma^*D_j=D'_{j,1}+D'_{j,2}.$$

Let $\mathcal{L}$ and $\mathcal{F}$ be the tautological line bundles on $\TQ$ corresponding to the generating Abelian differentials and the generating quadratic differentials respectively. By definition $$\mathcal{F}=\mathcal{L}^2.$$

We can express the relative canonical bundle class for the fibration $f'\colon S'\rightarrow \TQ$ as 
\begin{eqnarray*}
\omega_{S'/\TQ}=f'^*\mathcal{L}\otimes \mathcal{O}_{S'}\left(\sum_{
   d_j~{\rm even}
  }\frac{d_j}{2}(D'_{j,1}+D'_{j,2})+\sum_{
   d_j~{\rm odd}
  }(d_j+1)D_j'\right). 
\end{eqnarray*}
Similarly we can express the relative canonical bundle class for the fibration $f\colon S\rightarrow \TQ$ as
$$\omega^2_{S/\TQ}=f^*\mathcal{F}\otimes \mathcal{O}_S\Big(\sum d_jD_j\Big).$$

The double cover $\sigma\colon S'\rightarrow S$ is ramified exactly over the divisor 
$$B=\sum_{
   d_j~{\rm odd}
  }D_j$$
in $S$. From the cyclic cover construction (see e.g., \cite[\S 17]{BHPV}) there exists a Cartier divisor class $D$ in $S$ such that
$$ B=2D, \quad \sigma^{*}D = \sum_{
   d_j~{\rm odd}} D'_j,$$
$$\sigma_{*}\mathcal{O}_{S'} = \mathcal{O}_S \oplus \mathcal O_S(-D), \quad  \omega_{S'/\TQ} = \sigma^*(\omega_{S/\TQ} (D)). $$
It follows that 
\begin{eqnarray*}
f'_*\omega_{S'/\TQ} & = & f_{*} \sigma_{*} (\sigma^*(\omega_{S/\TQ}(D))\otimes \mathcal O_{S'}) \\
& = & f_* (\omega_{S/\TQ} (D) \otimes \sigma_{*}\mathcal{O}_{S'}) \\
& = & f_*(\omega_{S/\TQ}\oplus \omega_{S/\TQ}(D)) \\
& = & f_*\omega_{S/\TQ}\oplus f_*\omega_{S/\TQ}(D).
\end{eqnarray*}
Namely, $f'_*\omega_{S'/\TQ}$ decomposes as a direct sum of the $\sigma$-invariant part $f_*\omega_{S/\TQ}$ and the $\sigma$-anti-invariant part $f_*\omega_{S/\TQ}(D)$.

Below we will often use the following notation 
\begin{eqnarray*}
k_j=\left[\frac{d_j+1}{2}\right]. 
\end{eqnarray*}
Then we have
$$\sum k_j=\sum_{
   d_j~{\rm even}}
   \frac{d_j}{2}+\sum_{
   d_j~{\rm odd}
  } \frac{d_j+1}{2} =g_{\text{eff}}+g-1.$$

The following result characterizes the divisor class $D$ geometrically.  
\begin{proposition} 
\label{prop:omega(D)}
In the above setting 
$$\omega_{S/\TQ}(D) =f^*\mathcal{L}\otimes \mathcal{O}_S \left( \sum k_j D_j \right). $$
\end{proposition}
\begin{proof}
We have 
\begin{eqnarray*}
(\omega_{S/\TQ}(D))^{\otimes 2} &=& f^*\mathcal{F}\otimes\mathcal{O}_S\Big(\sum d_jD_j\Big)\otimes \mathcal{O}_S(B)  \\
 &=& \left(f^*\mathcal{L}\otimes \mathcal{O}_S \left(\sum k_j D_j\right)\right)^{\otimes 2}. 
\end{eqnarray*}
Hence there is a $2$-torsion divisor class $T$ in $S$ such that 
$$\omega_{S/\TQ}(D) =f^*\mathcal{L}\otimes \mathcal{O}_S \left( \sum k_j D_j + T \right). 
$$
It remains to show that $T = 0$.

The relation $\omega_{S'/\TQ} = \sigma^{*} (\omega_{S/\TQ}(D))$ implies that $\sigma^{*}\mathcal O_S(T) = \mathcal O_{S'}$. Pushing it forward under $\sigma$ yields the identification that $$\mathcal O_S(T) \oplus \mathcal O_S(T-D)= \mathcal O_S \oplus \mathcal O_S(-D).$$ 
Consider the induced map 
$\mathcal O_S(T)|_F\rightarrow \mathcal O_S(-D)|_F$ where $F$ is a fiber of $S$. Note that $\deg T|_F = 0$. If $B\neq 0$, then $\deg (-D)|_F < 0$, and hence the above induced map on $F$ is zero, which implies that $\mathcal O_S(T)\rightarrow \mathcal O_S(-D)$ is the zero map.  
Hence $\mathcal O_S(T)= \mathcal O_S$ and $T$ is the trivial divisor class. If $B=0$, since by assumption the canonical double cover is connected, then $T-D=\omega_{S/\TQ}\big(-\sum k_j D_j\big) \otimes f^{*}\mathcal L^{-1}$ is a non-trivial $2$-torsion and $(T-D)|_F$ is also a non-trivial $2$-torsion, for otherwise $q$ would be a global square on $F$. Therefore, the induced map 
$\mathcal O_S|_F\rightarrow \mathcal O_S(T-D)|_F$ is the zero map. Hence  
$\mathcal O_S = \mathcal O_S(T)$ and $T$ is the trivial divisor class. 
\end{proof}

For simplicity below we will denote by $p_j$ the intersection of the section $D_j$ with a fiber. The next result studies when the Hodge bundle on a family of quadratic differentials can split as a direct sum of line bundles. 

\begin{proposition}
\label{prop:split}
Let $B$ be a family of quadratic differentials in $\TQ(d_1,\ldots, d_m)$. If there exist integers $a_j$ such that $0\leq a_j\leq k_j$ for all $j$, $\sum a_j=g$, and $h^0\big(\sum a_jp_j\big)=1$ for every fiber of $S$ over $B$, then the invariant and anti-invariant parts of the Hodge bundle on $B$ split into the following direct sums of line bundles:  
$$f_* \omega_{S/B}=\bigoplus f_*\mathcal{O}_{a_jD_j}(\omega_{S/B}) = \bigoplus_j\bigoplus^{a_j-1}_{i=0}f_*\mathcal{O}_{D_j}(\omega_{S/B}-iD_j), $$
$$f_* \omega_{S/B}(D)= \mathcal{L}\oplus\left(\bigoplus \big(\mathcal{L}\otimes f_*\mathcal{O}_{(k_j-a_j)D_j}(k_jD_j)\big)\right) = \mathcal{L}\oplus\left(\bigoplus_j\bigoplus^{k_j}_{i=a_j+1}\big(\mathcal{L}\otimes f_*\mathcal{O}_{D_j}(iD_j)\big)\right).$$ 
Moreover, if $B$ is compact $n$-dimensional with an ample divisor $A$ such that $D_j^2 \cdot f^*A^{n-1} < 0$ in $S$, then the direct sums of graded quotients of the Harder--Narasimhan filtrations of $f_* \omega_{S/B}$ and $f_* \omega_{S/B}(D)$ with respect to $A$ coincide with the above direct sum splittings respectively.  
\end{proposition}
\begin{proof} 
Let $K$ be the canonical divisor class of a fiber $F$ over $B$. 
By using the assumption $h^0(\sum a_jp_j) = 1$ and Riemann--Roch, we have 
$h^0\big(K -\sum a_j p_j\big) = 0$, and 
hence $h^0\big(K-\sum (a_j+l_j) p_j\big) = 0$
for any $l_j\geq 0$. It follows that 
$$ f_*\omega_{S/B}\Big(-\sum a_jD_j\Big) = 0, $$
$$h^0\Big(\sum k_jp_j\Big)-\sum(k_j-a_j)=h^0\Big(\sum a_jp_j\Big) = 1. $$
We can thus apply Lemma \ref{sp2} to obtain that 
\begin{eqnarray*}
f_* \omega_{S/B} & = &  f_*\omega_{S/B}/f_*\omega_{S/B}\Big(-\sum a_jD_j\Big)\\ 
& = & \bigoplus f_*\mathcal{O}_{a_jD_j}(\omega_{S/B}). 
\end{eqnarray*}
By \cite[Lemma A.2]{CM2} we have 
$$ \quad f_*\mathcal{O}_S\Big(\sum a_jD_j\Big) = \mathcal O_B.$$
Hence we can apply Proposition~\ref{prop:omega(D)} and Lemma \ref{sp2} to obtain that 
\begin{eqnarray*}
    f_* \omega_{S/B}(D) / \mathcal L & = &  
\mathcal L \otimes \left( 
f_*\mathcal{O}_S\Big(\sum k_jD_j\Big)/f_*\mathcal{O}_S\Big(\sum a_jD_j\Big)\right) \\
& = & \bigoplus\big( \mathcal L \otimes 
 f_*\mathcal{O}_{(k_j-a_j)D_j}(k_jD_j)\big). 
 \end{eqnarray*}
By utilizing Lemma \ref{sp1} we can derive the desired expressions for splitting as the direct sums of the
graded quotients of the respective filtrations.  
\end{proof}

In the remaining part of this section we apply the preceding results to study Teichm\"uller curves and Lyapunov exponents. For the canonical double cover $(X, \omega)$ of $(Y, q)$ with the involution $\tau$, denote by $\lambda_i^+$ the Lyapunov exponents of the $\tau$-invariant part of $H^1(X,\mathbb{R})$, and $\lambda_i^-$ the Lyapunov exponents of the $\tau$-anti-invariant part. Define
$$L^+=\lambda_1^++\cdots +\lambda_g^+\quad {\rm and} \quad L^-=\lambda_1^-+\cdots +\lambda_{g_{\rm eff}}^- $$
where $g = g(Y)$ and $g+g_{\rm eff} = g(X)$. 
By \cite[Theorem 2]{EKZ} the following relation holds: 
\begin{equation*}
L^--L^+=\frac{1}{4}\sum_{
  d_j~{\rm odd} 
  }\frac{1}{d_j+2}.
\end{equation*}
Note that the $\tau$-invariant part descends to $Y$, and hence the Lyapunov exponents $\lambda_i^+$ of $(Y,q)$ are analogous to the ordinary Lyapunov exponents in the case of Abelian differentials. 

Next we recall some results about Teichm\"uller curves from \cite[Proposition 4.2]{CM2}. Denote by $\chi$ the Euler characteristic of a Teichm\"uller curve $C$ in $\TQ$, which also lifts as a Teichm\"uller curve via the canonical double covering construction in the corresponding stratum of Abelian differentials. We have 
$$\deg \mathcal L = \frac{\chi}{2} \quad {\rm and}\quad \deg \mathcal F = \chi $$
restricted to $C$. The self-intersection number of a section $D_j$ in the universal family $S$ over $C$ is 
$$D_j^2= -\frac{\chi}{d_j+2}.$$
For later use we also compute the degree of $f_*\mathcal{O}_{D_j}(\omega_{S/C}-iD_j)$ as follows:
\begin{eqnarray}\label{nu}
\deg f_*\mathcal{O}_{D_j}(\omega_{S/C}-iD_j) & = & -(i+1)D_j^2 \\ \nonumber
& = & \frac{\chi}{2} \cdot \frac{2i+2}{d_j+2}. 
\end{eqnarray}

Recall that $f'_*\omega_{S'/C}$ splits as a direct sum of the $\tau$-invariant part $f_*\omega_{S/C}$ and the $\tau$-anti-invariant part $f_*\omega_{S/C}(D)$. By \cite{EKZ, CM1, CM2} we have 
$$L^+(C)=\frac{\deg f_*(\omega_{S/C})}{\frac{1}{2}\chi}\quad {\rm and} \quad L^-(C)=\frac{\deg f_*\omega_{S/C}(D)}{\frac{1}{2}\chi}.$$
Define a $(g+g_{\rm eff})$-tuple of numbers representing the normalized Harder--Narasimhan polygon of the Hodge bundle $f'_*\omega_{S'/C}$ as 
$$w(C) = (w_1,\dots,w_{g+g_{\mathrm{eff}}}),$$ 
which is introduced in \cite[Definition 3.7]{YZ1}. Here the normalization mods out the factor $ \deg \mathcal L = \chi/2$, that is, $$w_i = \frac{\mu_i}{\frac{1}{2}\chi}$$ 
where the sequence of $\mu_i$ was defined  in \eqref{eq:mu_i} for the Harder--Narasimhan filtration of  $f'_*\omega_{S'/C}$. 
Since 
$$\text{grad}({\rm HN}(f'_*\omega_{S'/C}))=\text{grad}({\rm HN}(f_*\omega_{S/C}))\oplus \text{grad}({\rm HN}(f_*\omega_{S/C}(D))),$$
we can divide the $w_i$ into the following two subsets of numbers \cite{Yu}:
$$w_1^+\geq \cdots \geq w_g^+\quad\text{\,\,and\,\,} \quad 1=w_1^-\geq \cdots \geq w_{g_{\mathrm{eff}}}^-,$$
where $w_1^+,\dots,w_g^+$ and $w_1^-,\dots, w_{g_{\mathrm{eff}}}^-$ come respectively from the graded quotients $\text{grad}({\rm HN}(f_*\omega_{S/C}))$ and $\text{grad}({\rm HN}(f_*\omega_{S/C}(D)))$. By definition, we have
$$L^+(C)=w_1^+ +\cdots + w_g^+\quad \text{\,\,and\,\,}\quad L^-(C)=w_1^- +\cdots +w_{g_{\rm eff}}^-.$$
For any $k$, we also have the following inequalities between the Harder--Narasimhan polygon and the Lyapunov polygon \cite{Yu, EKMZ}:
\begin{equation}\label{lw}\sum_{i=1}^{k} \lambda_i^+\geq \sum_{i=1}^{k} w_i^+ \quad 
\text{\,\,and\,\,}\quad 
\sum_{i=1}^{k} \lambda_i^-\geq \sum_{i=1}^{k} w_i^-.
\end{equation}

Now we can apply the previously established results to the case of Teichm\"uller curves. 

\begin{proposition}
\label{prop:split-teich}
If there exist integers $a_j$ such that 
$0\leq a_j\leq k_j$ for all $j$, $\sum a_j=g$, and $h^0\big(\sum a_jp_j\big)=1$ for every fiber over the Teichm\"uller curve $C$, then 
$f_* \omega_{S/C}$ and $f_* \omega_{S/C}(D)$ split and coincide respectively with the direct sums of graded quotients of their Harder--Narasimhan filtrations: 
$$f_* \omega_{S/C} = \mathrm{grad\,}({\rm HN}(f_*\omega_{S/C}))=\bigoplus_j\bigoplus^{a_j-1}_{i=0}f_*\mathcal{O}_{D_j}(\omega_{S/C}-iD_j),$$
$$f_* \omega_{S/C}(D) = \mathrm{grad\,}({\rm HN}(f_*\omega_{S/C}(D)))=\mathcal{L}\oplus\left(\bigoplus_j\bigoplus^{k_j}_{i=a_j+1}\big(\mathcal{L}\otimes f_*\mathcal{O}_{D_j}(iD_j)\big)\right).$$
\end{proposition}
\begin{proof} 
Note that the self-intersection number $D_j^2 = - \chi/(d_j+2) < 0$ in the universal family $S$ over $C$ (see also \cite[Theorem (6.33)]{HM} for a general one-parameter family of curves). Then the claim follows as a special case of Proposition~\ref{prop:split}. 
\end{proof}

\section{Non-varying strata of quadratic differentials} 
\subsection{Genus zero}
We first prove Theorem~\ref{main} (1) for the case of the non-varying strata of quadratic differentials in genus zero. The following proof also contains a new computation of the normalized Harder--Narasimhan polygon on the hyperellitpic locus in a stratum of Abelian differentials that arises from the canonical double cover of quadratic differentials in genus zero. 
\begin{proposition}[{\cite[Theorem 5.6]{YZ1}}]
\label{hy} Let $C$ be a Teichm\"uller curve in the hyperelliptic locus of a stratum of Abelian differentials.  
Denote by $(d_1,\ldots,d_n)$ the orders of singularities of the underlying quadratic differentials in genus zero. Then the Hodge bundle on $C$ splits as a direct sum of line bundles which is equal to the direct sum of graded quotients of its Harder--Narasimhan filtration. Moreover, the normalized Weierstrass exponents of the Harder--Narasimhan polygon are given by 
$w_1 = w_1^-=1$ and $w_i=w_i^-$ for $i\geq 2$ as the $(i-1)$-st largest number in the set 
$$\left\{1-\frac{2k}{d_j+2}\ \mid \ 0<2k\leq d_j+1,\ j=1,\ldots,n\right\}.$$
\end{proposition}
\begin{proof}
Since the genus of the underlying curves is zero, we can take $a_j = 0$ for all $j$ and apply Proposition~\ref{prop:split-teich}. 
It follows that 
$$ f_*\omega_{S/C}(D) = \mathcal L\oplus \left( \bigoplus_j\bigoplus_{i=1}^{k_j} \big(\mathcal L\otimes f_{*}\mathcal O_{D_j}(iD_j)\big)\right), $$
which is also equal to the direct sum of graded quotients of the Harder--Narasimhan filtration. 
Finally for each line bundle $\mathcal{L}\otimes f_*\mathcal{O}_{D_j}(kD_j)$, we can compute its normalized degree as 
$$\frac{\deg (\mathcal{L}\otimes f_*\mathcal{O}_{D_j}(kD_j))}{\frac{1}{2}\chi} =1-\frac{2k}{d_j+2}.$$
This characterizes completely the Harder--Narasimhan polygon for Teichm\"uller curves in the hyperelliptic loci of Abelian differentials.
\end{proof}

\begin{remark}
\label{rem:g=0}
Comparing to \cite{YZ1}, a new result here is the splitting of the Hodge bundle as a direct sum of line bundles on the strata of quadratic differentials in genus zero (identified with the hyperelliptic loci of Abelian differentials). Moreover, the invariant part of the Hodge bundle is zero since the genus of the underlying quadratic differentials is zero, and hence there is no $w_i^+$ in this case. 
\end{remark}

\subsection{Irregular strata}
In this section we prove Theorem~\ref{main} (2) for the case of the non-varying irregular strata. 
\begin{proposition}\label{ir}
For Teichm\"uller curves in the irregular strata $\mathcal{Q}(9,-1)^{\rm irr}$, $\mathcal{Q}(6,3,-1)^{\rm irr}$, $\mathcal{Q}(12)^{\rm irr}$, and $\mathcal{Q}(9,3)^{\rm irr}$, the values of their $w_i^+$ and $w_i^-$ are given in Table \ref{irr}. 
\end{proposition}
\begin{proof} 
In the stratum $\mathcal{Q}(9,-1)^{\rm irr}$, we have $h^0(2p_1)=1$, $h^0(3p_1)=h^0(4p_1)=2$, and $h^0(5p_1)=3$. Note that these also hold for boundary points of Teichm\"uller curves in the stratum by using a similar analysis as in \cite[Section 6.3]{CM2}. We then obtain the Harder--Narasimhan filtration for the invariant part of the Hodge bundle as follows:
$$0=f_*\omega_{S/C}(-4D_1)\subset f_*\omega_{S/C}(-3D_1)=f_*\omega_{S/C}(-2D_1)\subset f_*\omega_{S/C}(-D_1)\subset f_*\omega_{S/C}.$$
By using \eqref{nu} the exponents of the invariant part of the Hodge bundle are given by 
$$w_1^+=\frac{8}{11},\quad w_2^+=\frac{4}{11}, \quad w_3^+=\frac{2}{11}.$$ 
Similarly the Harder--Narasimhan filtration for the anti-invariant part of the Hodge bundle is as follows:
$$0\subset\mathcal{L}\otimes f_*\mathcal{O}_S(2D_1)\subset\mathcal{L}\otimes f_*\mathcal{O}_S(3D_1)=\mathcal{L}\otimes f_*\mathcal{O}_S(4D_1)\subset \mathcal{L}\otimes f_*\mathcal{O}_S(5D_1)=f_*\omega_{S/C}(D).$$
 The corresponding exponents for this filtration are given by $$w_1^-=1, \quad w_2^-=1-\frac{6}{11}, \quad w_3^-=1-\frac{10}{11}. $$

In the stratum $\mathcal{Q}(6,3,-1)^{\rm irr}$, we have $h^0(p_1+p_2)=1$, $h^0(2p_1+p_2)=h^0(3p_1+p_2)=2$, and $h^0(3p_1+2p_2)=3$, which also hold for boundary points of Teichm\"uller curves in this stratum. We then obtain the Harder--Narasimhan filtration for the invariant part of the Hodge bundle as follows: $$0=f_*\omega_{S/C}(-3D_1-D_2)\subset f_*\omega_{S/C}(-2D_1-D_2)=f_*\omega_{S/C}(-D_1-D_2)\subset f_*\omega_{S/C}(-D_1)\subset  f_*\omega_{S/C}.$$ 
The exponents of the invariant part of the Hodge bundle are given by
 $$w_1^+=\frac{6}{8}, \quad w_2^+=\frac{2}{5}, \quad w_3^+=\frac{2}{8}.$$ 
Similarly the Harder--Narasimhan filtration for the anti-invariant part of the Hodge bundle is as follows:
$$0\subset\mathcal{L}\otimes f_*\mathcal{O}_S(D_1+D_2)\subset\mathcal{L}\otimes f_*\mathcal{O}_S(2D_1+D_2)=\mathcal{L}\otimes f_*\mathcal{O}_S(3D_1+D_2)\subset \mathcal{L}\otimes f_*\mathcal{O}_S(3D_1+2D_2)=f_*\omega_{S/C}(D).$$
 The corresponding exponents of this filtration are given by 
 $$w_1^-=1, \quad w_2^-=1-\frac{4}{8}, \quad w_3^-=1-\frac{4}{5}. $$

In the stratum $\mathcal{Q}(12)^{\rm irr}$, we have $h^0(3p_1)=1$, $h^0(4p_1)=h^0(5p_1)=2$, and $h^0(6p_1)=3$, which also hold for boundary points of Teichm\"uller curves in this stratum. We then obtain the Harder--Narasimhan filtration for the invariant part of the Hodge bundle as follows: $$0=f_*\omega_{S/C}(-5D_1)\subset f_*\omega_{S/C}(-4D_1)= f_*\omega_{S/C}(-3D_1)\subset f_*\omega_{S/C}(-2D_1)\subset f_*\omega_{S/C}(-D_1)\subset f_*\omega_{S/C}.$$
The corresponding exponents of the filtration are given by
 $$w_1^+=\frac{10}{14},\quad w_2^+=\frac{6}{14}, \quad w_3^+=\frac{4}{14}, \quad w_4^+=\frac{2}{14}.$$
 Similarly the Harder--Narasimhan filtration for the anti-invariant part of the Hodge bundle is as follows:
$$0\subset\mathcal{L}\otimes f_*\mathcal{O}_S(3D_1)\subset\mathcal{L}\otimes f_*\mathcal{O}_S(4D_1)=\mathcal{L}\otimes f_*\mathcal{O}_S(5D_1)\subset \mathcal{L}\otimes f_*\mathcal{O}_S(6D_1)=f_*\omega_{S/C}(D).$$
 The corresponding exponents of this filtration are given by 
 $$w_1^-=1, \quad w_2^-=1-\frac{8}{14}, \quad w_3^-=1-\frac{12}{14}. $$

In the stratum $\mathcal{Q}(9,3)^{\rm irr}$, we have $h^0(2p_1+p_2)=1$, $h^0(3p_1+p_2)=h^0(4p_1+p_2)=2$, and $h^0(4p_1+2p_2)=3$, which also hold for boundary points of Teichm\"uller curves in this stratum. We then obtain the Harder--Narasimhan filtration for the invariant part of the Hodge bundle as follows: 
$$0=f_*\omega_{S/C}(-4D_1-D_2)\subset f_*\omega_{S/C}(-3D_1-D_2)= f_*\omega_{S/C}(-2D_1-D_2)\subset $$
$$\subset f_*\omega_{S/C}(-2D_1)\subset f_*\omega_{S/C}(-D_1)\subset  f_*\omega_{S/C}.$$
The corresponding exponents of this filtration are given by
 $$w_1^+=\frac{8}{11},\quad w_2^+=\frac{2}{5},\quad w_3^+=\frac{4}{11},\quad w_4^+=\frac{2}{11}.$$ Similarly the Harder--Narasimhan filtration for the anti-invariant part of the Hodge bundle is as follows:
$$0\subset\mathcal{L}\otimes f_*\mathcal{O}_S(2D_1+D_2)\subset\mathcal{L}\otimes f_*\mathcal{O}_S(3D_1+D_2)=\mathcal{L}\otimes f_*\mathcal{O}_S(4D_1+D_2)\subset $$
$$\subset \mathcal{L}\otimes f_*\mathcal{O}_S(4D_1+2D_2)\subset \mathcal{L}\otimes f_*\mathcal{O}_S(5D_1+2D_2)=f_*\omega_{S/C}(D).$$
 The corresponding exponents of this filtration are given by 
 $$w_1^-=1,\quad w_2^-=1-\frac{6}{11},\quad w_3^-=1-\frac{4}{5},\quad w_4^-=1-\frac{10}{11}.$$
\end{proof}

\begin{remark}
\label{rem:irregular}
The Hodge bundle on the irregular strata does not necessarily split into a direct sum of line bundles. For example for $\mathcal{Q}(9,-1)^{\rm irr}$, we have $h^0(3p_1) = 2\neq 1$, hence Proposition~\ref{prop:split-teich} is not applicable in this case. However, for these irregular strata we can still obtain partial splitting results. For example for $\TQ = \TQ(9,-1)^{\rm irr}$, we have $h^1(K-2p_1) = h^1(K) = 1$, hence $f_{*}\omega_{S/\TQ} / f_{*}\omega_{S/\TQ}(-2D_1) = f_{*} \mathcal O_{2D_1}(\omega_{S/\TQ})$ splits as a direct sum of two line bundles. 
\end{remark}

\begin{table}\label{irr}
\small
$$
\begin{array}{|c|c|c|c|c|c||c|c|c|}

\hline
&&\multicolumn{4}{|c||}{}&\multicolumn{3}{|c|}{}\\

\multicolumn{1}{|c|}{\text{Degrees}}&
\multicolumn{1}{|c|}{\text{Con-}}&
\multicolumn{4}{|c||}{\text{Invariant }}&
\multicolumn{3}{|c|}{\text{Anti-invariant}}
\\

\multicolumn{1}{|c|}{\text{of }}&
\multicolumn{1}{|c|}{\text{nected}}&
\multicolumn{4}{|c||}{\text{part}}&
\multicolumn{3}{|c|}{\text{part}}
\\
\cline{3-9}

\multicolumn{1}{|c|}{\text{zeros}}&
\multicolumn{1}{|c|}{\text{comp.}}&
\multicolumn{1}{|c|}{w_1^+}&
\multicolumn{1}{|c|}{w_2^+}&
\multicolumn{1}{|c|}{w_3^+}&
\multicolumn{1}{|c||}{w_4^+}&
\multicolumn{1}{|c|}{w_2^-}&
\multicolumn{1}{|c|}{w_3^-}&
\multicolumn{1}{|c|}{w_4^-}
\\
[-\halfbls] &&&&&&&&\\
\hline &&&&&&&& \\ [-\halfbls]
(9,-1) & {\rm Irr} &
\frac{8}{11} & \frac{4}{11} & \frac{2}{11} & -
&\frac{5}{11} & \frac{1}{11} & -\\
[-\halfbls] &&&&&&&&\\
\hline &&&&&&&& \\ [-\halfbls]
(6,3,-1) & {\rm Irr} &
 \frac{3}{4} & \frac{2}{5} &  \frac{1}{4} & - & \frac{1}{2} & \frac{1}{5} & -
\\
[-\halfbls] &&&&&&&&\\
\hline &&&&&&&& \\ [-\halfbls]
(12) & {\rm Irr} &
\frac{5}{7} &  \frac{3}{7} &  \frac{2}{7} &  \frac{1}{7} &  \frac{3}{7} &  \frac{1}{7} & -
\\
[-\halfbls] &&&&&&&&\\
\hline &&&&&&&& \\ [-\halfbls]
(9,3) & {\rm Irr} &
 \frac{8}{11} & \frac{2}{5} &  \frac{4}{11} &  \frac{2}{11} & \frac{5}{11} & \frac{1}{5} & \frac{1}{11}  
\\
[-\halfbls] &&&&&&&&\\ \hline
\end{array}
$$
\caption{Non-varying irregular strata}
\end{table}

\subsection{Genus one to four}
We have treated the cases of genus zero and the irregular strata. In this section we will prove for the other strata mentioned in Theorem~\ref{main} (1). 
In what follows we will often refer to the dimension of a linear system that has been computed in \cite{CM2}, which indeed holds for all points in the partial compactification $\TQ$ of the respective strata.  

\subsubsection{Genus one} For the case $g=1$, we have $h^0(p_1)=1$. Consider first the stratum $\mathcal{Q}(n,-1^n)$. Then we can apply Proposition \ref{prop:split-teich} with $k_1=[\frac{n+1}{2}]$, $k_i=0$ for $i\geq 2$, $a_1=1$, and $a_i=0$ for $i\geq 2$, which implies the following expressions:
$$f_*\omega_{S/C}=f_*\mathcal{O}_{D_1}(\omega_{S/C}),$$
$$f_*\omega_{S/C}(D)= \mathcal{L}\oplus \big(\mathcal{L}\otimes f_*\mathcal{O}_{(k_1-1)D_1}(k_1D_1)\big) =\mathcal L \oplus\left(\bigoplus^{k_1}_{i=2}\big(\mathcal{L}\otimes f_*\mathcal{O}_{D_1}(iD_1)\big)\right).$$
From these expressions we obtain that 
$$w_1^+=\frac{2}{n+2}.
$$ 
Furthermore, $w_1^-=1$ and $w_i^-$ for $i\geq 2$ are given by the $(i-1)$-st largest number in the set
$$\left\{1-\frac{2k}{n+2}\ |\ 2<2k\leq n+1\right\}.$$

Next consider the stratum $\mathcal{Q}(n,1,-1^{n+1})$. We can apply Proposition \ref{prop:split-teich} with $k_1=[\frac{n+1}{2}]$, $k_2=1$, $k_i=0$ for $i\geq 3$, $a_1=1$, and $a_i=0$ for $i\geq 2$, which implies the following expressions:
$$f_*\omega_{S/C}=f_*\mathcal{O}_{D_1}(\omega_{S/C}),$$
 \begin{eqnarray*}f_*\omega_{S/C}(D) &= & \mathcal{L}\oplus \big(\mathcal{L}\otimes f_*\mathcal{O}_{(k-1)D_1}(kD_1)\big)\oplus \big(\mathcal{L}\otimes f_*\mathcal{O}_{D_2}(D_2)\big) \\
 & = & \mathcal{L}\oplus\left(\bigoplus^{k}_{i=2} \big(\mathcal{L}\otimes f_*\mathcal{O}_{D_1}(iD_1)\big)\right)\oplus \big(\mathcal{L}\otimes f_*\mathcal{O}_{D_2}(D_2)\big).
 \end{eqnarray*}
 From these expressions we obtain that 
 $$w_1^+=\frac{2}{n+2}.
 $$ 
 Furthermore, $w_1^-=1$ and $w_i^-$ for $i\geq 2$ are given by the $(i-1)$-st largest number in the set 
$$\frac{1}{3}, \ \left\{1-\frac{2k}{n+2}\ |\ 2<2k\leq n+1\right\}.$$

As an application we can study the asymptotic behavior of Lyapunov exponents in these non-varying strata of genus one.  

\begin{corollary}
\label{as} Given any positive integer $m$, 
for the strata $\mathcal{Q}(n,-1^n)$ and  $\mathcal{Q}(n,1,-1^{n+1})$ we have $$\lim_{n\to\infty}\lambda_m^-=1. $$ 
\end{corollary}
\begin{proof} 
For $n\geq 4$, all Teichm\"uller curves in $\mathcal{Q}(n,-1^n)$ and $\mathcal{Q}(n,1,-1^{n+1})$ have 
$$w_2^-=\frac{n-2}{n+2}.$$ 
By the inequality \eqref{lw} and the fact that $\lambda_1^- = w_1^- = 1$, the Lyapunov exponents of all Teichm\"uller curves in these strata satisfy that 
$$\lambda_2^-\geq w^-_2=\frac{n-2}{n+2}.$$
Since the union of Teichm\"uller curves is dense in every stratum, by \cite{BEW} the Lyapunov exponents of these strata satisfy the same inequality  
$$\lambda_2^-\geq\frac{n-2}{n+2}.$$ 
Therefore, we conclude that  
$$1\geq\lim_{n\to\infty}\lambda_2^-\geq \lim_{n\to\infty}\frac{n-2}{n+2}=1.$$
The proof for $\lambda_m^-$ with a general given value of $m$ is completely analogous.
\end{proof}

\subsubsection{Genus two} For the case $g=2$, we demonstrate by studying the stratum $\mathcal{Q}(7,-1,-1,-1)$ in detail.  
In this case $h^0(2p_1)=1$ and we can apply Proposition \ref{prop:split-teich} with $k_1=4$, $k_2=k_3=k_4=0$, $a_1=2$, and $a_2=a_3=a_4=0$. We then obtain the following expressions: 
$$f_*\omega_{S/C} =f_*\mathcal{O}_{D_1}(\omega_{S/C})\oplus f_*\mathcal{O}_{D_1}(\omega_{S/C}-D_1),$$
\begin{eqnarray*}
f_*\omega_{S/C}(D) & = & \mathcal{L}\oplus\big(\mathcal{L}\otimes f_*\mathcal{O}_{2D_1}(4D_1)\big) \\
& = & \mathcal{L}\oplus \big(\mathcal{L}\otimes f_*\mathcal{O}_{D_1}(3D_1)\big)\oplus\big(\mathcal{L}\otimes f_*\mathcal{O}_{D_1}(4D_1)\big).
\end{eqnarray*}
From these expressions we obtain that 
$$w_1^+=\frac{4}{9},\quad w_2^+=\frac{2}{9}.
$$ 
$$ w_1^-=1,\quad w_2^-=1-\frac{6}{9},\quad w_3^-=1-\frac{8}{9}. 
$$ 
The computations for the strata $\mathcal{Q}(5,-1)$, $\mathcal{Q}(6,-1,-1)^{\nonhyp}$, and $\mathcal{Q}(5,1,-1,-1)$ can be carried out in the same manner and we skip the details. 

For the remaining strata $\mathcal{Q}(3,2,-1)$, $\mathcal{Q}(2,2,1,-1)$, $\mathcal{Q}(4,1,-1)$, $\mathcal{Q}(4,2,-1,-1)$, $\mathcal{Q}(3,1,1,-1)$, $\mathcal{Q}(3,3,-1,-1)^{\text{non-hyp}}$,  $\mathcal{Q}(3,2,1,-1,-1)$, and $\mathcal{Q}(4,3,-1,-1,-1)$, we can use $h^0(p_1+p_2)=1$ and perform the computations analogously. 

\subsubsection{Genus three} For the case $g=3$, we demonstrate by studying the stratum $\mathcal{Q}(4,3,2,-1)$ in detail. 
In this case $h^0(p_1+p_2+p_3)=1$ and we can apply Proposition \ref{prop:split-teich} with $k_1=k_2=2$, $k_3=1$, $k_4=0$, $a_1=a_2=a_3=1$, and $a_4=0$. We then obtain the following expressions:
$$f_*\omega_{S/C}=f_*\mathcal{O}_{D_1}(\omega_{S/C})\oplus f_*\mathcal{O}_{D_2}(\omega_{S/C})\oplus f_*\mathcal{O}_{D_3}(\omega_{S/C}),$$
$$f_*\omega_{S/C}(D)=\mathcal{L}\oplus \big(\mathcal{L}\otimes f_*\mathcal{O}_{D_1}(2D_1)\big)\oplus \big(\mathcal{L}\otimes f_*\mathcal{O}_{D_2}(2D_2)\big).$$ 
From these expressions we obtain that 
$$w_1^+=\frac{2}{4},\quad w_2^+=\frac{2}{5},\quad w_3^+=\frac{2}{6}, $$  
$$w_1^-=1,\quad w_2^-=1-\frac{4}{6},\quad w_3^-=1-\frac{4}{5}. $$ 
The computations for the strata $\mathcal{Q}(4,2,2)$, $\mathcal{Q}(3,3,2)^{\nonhyp}$, $\mathcal{Q}(3,3,1,1)^{\nonhyp}$, $\mathcal{Q}(4,3,1)$,  $\mathcal{Q}(3,2,2,1)$, and $\mathcal{Q}(3,3,3,-1)^{\rm reg}$ can be carried out similarly.

For the strata $\mathcal{Q}(6,2)^{\nonhyp}$, $\mathcal{Q}(4,4)$, $\mathcal{Q}(5,3)$, $\mathcal{Q}(5,2,1)$, $\mathcal{Q}(6,1,1)^{\nonhyp}$, $\mathcal{Q}(5,4,-1)$, $\mathcal{Q}(5,3,1,-1)$, $\mathcal{Q}(7,1)$, $\mathcal{Q}(7,2,-1)$, $\mathcal{Q}(7,3,-1,-1)$, and $\mathcal{Q}(6,3,-1)^{\rm reg}$, we can use $h^0(2p_1+p_2) = 1$. For the strata $\mathcal{Q}(9,-1)^{\rm reg}$, $\mathcal{Q}(8)$, $\mathcal{Q}(8,1,-1)$, and $\mathcal{Q}(10,-1,-1)^{\nonhyp}$, we can use $h^0(3p_1)=1$. 

\subsubsection{Genus four} For the case $g=4$, we demonstrate by studying the stratum $\mathcal{Q}(5,4,3)$ in detail. In this case $h^0(2p_1+p_2+p_3)=1$ and we can apply Proposition \ref{prop:split-teich} with $k_1=3$, $k_2=k_3=2$, $a_1=2$, and $a_2=a_3=1$. We then obtain the following expressions:
$$f_*\omega_{S/C}=f_*\mathcal{O}_{D_1}(\omega_{S/C})\oplus f_*\mathcal{O}_{D_1}(\omega_{S/C}-D_1)\oplus f_*\mathcal{O}_{D_2}(\omega_{S/C})\oplus f_*\mathcal{O}_{D_3}(\omega_{S/C}),$$ 
$$f_*\omega_{S/C}(D)=\mathcal{L}\oplus\big( \mathcal{L}\otimes f_*\mathcal{O}_{D_1}(3D_1)\big)\oplus \big(\mathcal{L}\otimes f_*\mathcal{O}_{D_2}(2D_2)\big)\oplus\big( \mathcal{L}\otimes f_*\mathcal{O}_{D_3}(2D_3)\big).$$
From these expressions we obtain that 
$$w_1^+=\frac{4}{7},\quad w_2^+=\frac{2}{5},\quad w_3^+=\frac{2}{6},\quad w_4^+=\frac{2}{7}, $$ 
$$w_1^-=1,\quad w_2^-=1-\frac{4}{6},\quad w_3^-=1-\frac{4}{5},\quad w_4^-=1-\frac{6}{7}.$$ 
The computations for the strata $\mathcal{Q}(7,3,2)$ and $\mathcal{Q}(6,3,3)^{\rm reg}$ can be carried out similarly.

For the strata $\mathcal{Q}(13,-1)$, $\mathcal{Q}(11,1)$, and $\mathcal{Q}(12)^{\rm reg}$ we can use $h^0(4p_1)=1$. For the strata $\mathcal{Q}(8,3,1)$, $\mathcal{Q}(10,2)^{\nonhyp}$, $\mathcal{Q}(8,4)$, $\mathcal{Q}(10,2)^{\nonhyp}$, and $\mathcal{Q}(9,3)^{\rm reg}$ we can use $h^0(3p_1+p_2)=1$. For the strata $\mathcal{Q}(7,5)$ and $\mathcal{Q}(6,6)^{\rm reg}$ we can use $h^0(2p_1+2p_2)=1$. For the stratum $\mathcal{Q}(3,3,3,3)^{\rm reg}$ we can use $h^0(p_1+p_2+p_3+p_4)=1$.  

\subsubsection{Summary} The following result holds for each of the above strata in genus one to four (with the exception of the irregular strata treated in the previous section). For $i=1, \ldots, g$, the exponent $w_i^+$ is the $(g-i+1)$-st smallest number in the set 
$$\left\{\frac{2k}{d_j+2}\ \mid\ 0< 2k\leq d_j+1,\ j=1,\ldots,n\right\}.$$ 
Moreover, $w_1^-=1$ and $w_i^-$ for $i=2,\ldots,g_{\text{eff}}$ is the $(g_{\text{eff}}-i+1)$-st smallest number in the set
    $$\left\{1-\frac{2k}{d_j+2}\ \mid\ 0<2k\leq d_j+1,\ j=1,\ldots,n\right\}.$$
Finally in all these cases the Hodge bundle on Teichm\"uller curves splits into a direct sum of line bundles by using Proposition~\ref{prop:split-teich}.
Indeed the direct sum splitting of the Hodge bundle holds more generally for each of these strata by using Proposition~\ref{prop:split}. 

\subsection{Non-varying strata of Abelian differentials} For the sake of completeness we also consider the non-varying strata of Abelian differentials appearing in \cite{CM1}. The Harder--Narasimhan filtration and Weierstrass exponents of the Hodge bundle for these non-varying strata were determined in \cite{YZ1}. Here we focus on the splitting of the Hodge bundle.  

\begin{theorem}
\label{thm:abelian-splitting}
The Hodge bundle on any non-varying stratum of Abelian differentials in \cite{CM1} that is not a spin component of even parity splits as a direct sum of line bundles. 
\end{theorem}

\begin{proof}
The proof is analogous to the preceding cases. Since we have treated the hyperelliptic strata of Abelian differentials, below we consider strata that are non-hyperelliptic and are not of even spin parity. Since the Hodge bundle contains the tautological line bundle $\mathcal L$ as a direct summand, by using the idea of  Proposition~\ref{prop:split} it suffices to find integers $a_j \geq 0$ such that $\sum a_j = g-1$ and $h^0(\sum a_j p_j) = 1$ for every Abelian differential in a concerned stratum. For $\HH(4)^{\rm odd}$ we can take $a_1 = 2$. For $\HH(3,1)$ and $\HH(2,2)^{\rm odd}$ we can take $a_1 = a_2 = 1$. For $\HH(2,1,1)$ we can take $a_1 = a_2 = 1$ and $a_3 = 0$. For $\HH(6)^{\rm odd}$ we can take $a_1 = 3$. 
For $\HH(5,1)$, $\HH(4,2)^{\rm odd}$, and $\HH(3,3)^{\operatorname{non-hyp}}$ we can take $a_1 = 2$ and $a_2 = 1$. For $\HH(2,2,2)^{\rm odd}$ and $\HH(3,2,1)$ we can take $a_1 = a_2 = a_3 = 1$. For $\HH(8)^{\rm odd}$ we can take $a_1 = 4$. Finally for $\HH(6,2)^{\rm odd}$ and $\HH(5,3)$ we can take $a_1 = 3$ and $a_2 = 1$. 
\end{proof}

\section{Hyperelliptic strata of quadratic differentials}
In this section we prove Theorem~\ref{main} (3) for the case of the hyperelliptic strata of quadratic differentials. To distinguish it from the hyperelliptic loci of Abelian differentials, instead of considering the canonical double cover, we start with $(\mathbb P^1, q)$ belonging to a stratum of quadratic differentials in genus zero, specify the topology of a hyperelliptic double cover $Y\to \mathbb P^1$ with branch points located at some zeros and poles of $q$, and then pull back $q$ to obtain a quadratic differential $(Y, q_Y)$. The main result established in \cite{La} asserts that the dimensions of the strata containing $(\mathbb P^1, q)$ and $(Y, q_Y)$ are only equal in the following three cases: 

\begin{itemize}
\item[1)] The hyperelliptic stratum $\mathcal{Q}_g(4(g-k)-6,4k+2)^{\hyp}$ is obtained by pulling back quadratic differentials from the stratum $\mathcal{Q}_0(2(g-k)-4,2k,-1^{2g})$, where $k\geq 0$, $g\geq 2$, and $g-k\geq 2$. The corresponding double cover is ramified over all the singularities. 
\item[2)] The hyperelliptic stratum $\mathcal{Q}_g(2(g-k)-3,2(g-k)-3,4k+2)^{\hyp}$ is  obtained by pulling back quadratic differentials from the stratum $\mathcal{Q}_0(2(g-k)-3,2k,-1^{2g+1})$, where $k\geq 0$, $g\geq 1$, and $g-k \geq 1$. The corresponding double cover is ramified over the $2g+1$ poles and over the zero of order $2k$.
 \item[3)] The hyperelliptic stratum $\mathcal{Q}_g(2(g-k)-3,2(g-k)-3,2k+1,2k+1)^{\hyp}$, is obtained by pulling back quadratic differentials from the stratum $\mathcal{Q}_0(2(g-k)-3,2k+1,-1^{2g+2})$, where $k\geq -1$, $g\geq 1$, and $g-k\geq 2$. The corresponding double cover is ramified over the $2g+2$ poles.
\end{itemize}

Consider the following diagram
\begin{eqnarray}
\label{eq:triple}
\mathcal{Q}_0(\ldots,d^{\rm o}_i,\ldots,d^{\rm o}_j,\ldots)\overset{\psi}{\longrightarrow} \mathcal{Q}_g(\ldots,d^{\rm o}_i,d^{\rm o}_i,\ldots,2d^{\rm o}_j+2,\ldots)\overset{\phi}{\longrightarrow} \mathcal{H}_{g+g_{\rm eff}}(\cdots) 
\end{eqnarray}
where $\psi$ is induced by the above hyperelliptic double cover, $\phi$ is induced by the canonical double cover, and both can extend to the partial compactification $\TQ$. Let $B^{\rm o}$ be the branch locus of the hyperelliptic double cover where $B^{\rm o}$ is contained in the union of the singularities of $q$. Denote by $D^{\rm o}_i \not\subset B^{\rm o}$ and $D^{\rm o}_j\subset B^{\rm o}$. We thus obtain three universal families, denoted by $f^{\rm o}$, $f'$, and $f$ respectively, in the following commutative diagram: 
$$\xymatrix{
    S^{\rm o}  \ar[dr]_{f^{\rm o}}
                & \ar[l]_{\theta} S \ar[d]^{f}  &  \ar[l]_{\sigma}  S' \ar[dl]^{f'}    \\
                & \TQ                }
$$
where we identify the images of $\psi$ and $\phi \circ \psi$, both denoted by $\TQ$.  

As before there exists a Cartier divisor class $D^{\rm o}$ in $S^{\rm o}$ such that $B^{\rm o} =2D^{\rm o}$, where 
the family of the hyperelliptic double covers $\theta\colon S\rightarrow S^{\rm o}$ is ramified exactly over $B^{\rm o}$. Let $D_j$ be the section of $f\colon S\rightarrow \TQ$ over $D^{\rm o}_j$ if $D^{\rm o}_j$ lies in $B^{\rm o}$, and let $D_{j,1}$ and $D_{j,2}$ be the sections over $D^{\rm o}_j$ if $D^{\rm o}_j$ does not lie in $B^{\rm o}$. In the case when $D^{\rm o}_j$ lies in $B^{\rm o}$, we have 
$$\theta_*D_j=D^{\rm o}_j, \quad \theta^*D^{\rm o}_j=2D_j. $$
In the case when $D_j^{\rm o}$ does not lie in $B^{\rm o}$, we have 
$$\theta_*(D_{j,1}+D_{j,2})=2D^{\rm o}_j, \quad \theta^*D^{\rm o}_j=D_{j,1}+D_{j,2}. $$

Recall that $\mathcal{F}$ and $\mathcal L$ are the tautological line bundles on $\TQ$ corresponding to the generating quadratic differentials and Abelian differentials respectively, where $\mathcal{F} = \mathcal L^2$. The relative canonical bundle class for the family $f^{\rm o}\colon S^{\rm o}\rightarrow \TQ$ satisfies that  
$$\omega^2_{S^{\rm o}/\TQ}=(f^{\rm o})^*\mathcal{F}\otimes \mathcal{O}_{S^{\rm o}}\left(\sum d^{\rm o}_jD^{\rm o}_j\right). $$
 The relative canonical bundle class for the family $f\colon S\rightarrow \TQ$ satisfies that  
$$\omega_{S/\TQ}^2=f^*\mathcal{F}\otimes \mathcal{O}_{S}\left(\sum_{
   D^{\rm o}_j \not\subset B
  }d^{\rm o}_j(D_{j,1}+D_{j,2})+\sum_{
   D^{\rm o}_j \subset B
  }2(d^{\rm o}_j+1)D_j\right)=f^*\mathcal{F}\otimes \mathcal{O}_S\left(\sum d_jD_j\right).$$

Since as a rational curve fibration $\omega_{S^{\rm o}/\TQ}$ has relative degree $-2$, we have 
$$ f^{\rm o}_*\omega_{S^{\rm o}/\TQ} = 0. $$
It follows that 
\begin{eqnarray}
\label{eq:o-invariant}
f_*\omega_{S/\TQ}  =  f^{\rm o}_*\omega_{S^{\rm o}/\TQ}\oplus f^{\rm o}_*\omega_{S^{\rm o}/\TQ}(D^{\rm o}) 
 =  f^{\rm o}_*\omega_{S^{\rm o}/\TQ}(D^{\rm o})
\end{eqnarray}
and hence 
\begin{eqnarray*}
\text{grad}({\rm HN}(f_*\omega_{S/\TQ})) =  \text{grad}({\rm HN}(f^{\rm o}_*\omega_{S^{\rm o}/\TQ}(D^{\rm o}))).
\end{eqnarray*}

Finally restricting to a Teichm\"uller curve $C$ in $\TQ$, we have the self-intersection number 
$$(D^{\rm o}_j)^2 = - \frac{\chi}{d^{\rm o}_j+2}$$
in $S^{\rm o}$ over $C$, where $\chi$ is the orbifold Euler characteristic of $C$. 

In what follows we first treat the anti-invariant part of the Hodge bundle. As before we introduce and use the notation 
\begin{eqnarray*}
k_i^{\rm o} = \left[\frac{d_i^{\rm o} + 1}{2}\right]. 
\end{eqnarray*}

\begin{theorem}\label{hylc} Let $\TQ$ be the hyperelliptic locus of quadratic differentials in genus $g$ induced from the stratum $\TQ_0(\ldots,d^{\rm o}_i,\ldots,d^{\rm o}_j,\ldots)$. Suppose the following condition holds for the generic fiber of the universal family $S$ over $\TQ$: 
$$h^0\left(\sum_{D^{\rm o}_i\not\subset B^{\rm o}} k_i^{\rm o}(p_{i,1}+p_{i,2})+\sum_{D^{\rm o}_j\subset B^{\rm o}}(d^{\rm o}_j+1)p_j\right)=h^0\left(\sum_{D^{\rm o}_i\not\subset B^{\rm o}} k_i^{\rm o}(p_{i,1}+p_{i,2})+\sum_{D^{\rm o}_j\subset B^{\rm o}}k_j^{\rm o}(2p_j)\right)$$ 
where $p_j$ is the corresponding ramified Weierstrass point in the hyperelliptic double cover and $p_{i,1}, p_{i,2}$ are the corresponding unramified conjugate points. Moreover, suppose 
$$\sum k_i^{\rm o}\leq g.$$
Then the anti-invariant part of the Hodge bundle over $\TQ$ splits into a direct sum of line bundles: 
$$f_*\omega_{S/\TQ}(D)=\mathcal{L}\oplus\left(\bigoplus_i\bigoplus^{k_i^{\rm o}}_{j=1} \big(\mathcal{L}\otimes f^{\rm o}_*\mathcal{O}_{D^{\rm o}_i}(jD^{\rm o}_i)\big)\right).$$

If the above assumptions on $\TQ$ hold for a Teichm\"uller curve $C$ in $\TQ$, then the same splitting of $f_*\omega_{S/C}(D)$ holds and coincides with the direct sum of graded quotients of the Harder--Naramsihan filtration, whose normalized Weierstrass exponents are given by 
$w_1^-=1$ and $w_i^-$ for $i\geq 2$ as the $(i-1)$-st largest number in the set 
$$\bigcup_i\left\{1-\frac{2l}{d^{\rm o}_i+2}\ \mid \ 0<l\leq \Big[\frac{d^{\rm o}_i+1}{2}\Big]\right\}.$$
\end{theorem}
\begin{proof}
Let $D_{i,1}$, $D_{i,2}$, and $D_j$ be the divisors in $S$ corresponding respectively to $p_{i,1}$, $p_{i,2}$, and $p_j$. By the first assumption and Proposition~\ref{prop:omega(D)}, we have 
\begin{eqnarray*}
f_*\omega_{S/\TQ}(D) & = & \mathcal{L}\otimes f_*\mathcal{O}_S\left(\sum_{D^{\rm o}_i\not\subset B^{\rm o}} k^{\rm o}_i(D_{i,1}+D_{i,2})+\sum_{D^{\rm o}_j\subset B^{\rm o}}(d^{\rm o}_j+1)D_j\right)\\
& = & \mathcal{L}\otimes f_*\mathcal{O}_S\left(\sum_{D^{\rm o}_i\not\subset B^{\rm o}} k^{\rm o}_i(D_{i,1}+D_{i,2})+\sum_{D^{\rm o}_j\subset B^{\rm o}}k^{\rm o}_j (2D_j)\right) \\
& = &\mathcal{L}\otimes f^{\rm o}_*\theta_*\mathcal{O}_S\left(\theta^*\Bigg(\sum_{D^{\rm o}_i\not\subset B^{\rm o}} k^{\rm o}_i D^{\rm o}_i+\sum_{D^{\rm o}_j\subset B^{\rm o}} k^{\rm o}_j D^{\rm o}_j\Bigg)\right) \\
& = & \mathcal{L}\otimes f^{\rm o}_* \left(\big(\mathcal{O}_{S^{\rm o}}\oplus \mathcal{O}_{S^{\rm o}}(-D^{\rm o})\big) \otimes \mathcal{O}_{S^{\rm o}}\Big(\sum k^{\rm o}_i D^{\rm o}_i\Big)\right) \\
& = & \mathcal{L}\otimes f^{\rm o}_*\mathcal{O}_{S^{\rm o}}\left(\sum k^{\rm o}_i D^{\rm o}_i\right)
\end{eqnarray*}
where the last equality follows from the second assumption that 
$$\deg \left(\sum k_i^{\rm o}D_i^{\rm o}\right)\Big|_{F^{\rm o}} < g +1 = \deg D^{\rm o}|_{F^{\rm o}}$$ 
for a fiber $F^{\rm o}$ of $S^{\rm o}$. We also remark that the first equality assumption on the generic fiber (as opposed to on every fiber) is sufficient for our purpose. To see it, note that the concerned direct image sheaves in the first two lines of the above displayed formula are subsheaves of a vector bundle (e.g., by taking the coefficients of the sections large enough), hence they are also vector bundles. Therefore, being equal on the generic fiber ensures that the two vector bundles are equal as one contains the other.    

The desired claims on the splitting and filtration thus follow from Lemma~\ref{sp2} by using $M= \sum k_i^{\rm o} D_i^{\rm o} $ and $a_i = k_i^{\rm o}$. The values of Weierstrass exponents of $C$ can be computed by using the self-intersection number $(D^{\rm o}_j)^2$ over $C$. 
\end{proof}

Next we treat the invariant part of the Hodge bundle. We introduce and use the notation 
$$ \ell_j^{\rm o} = \left[\frac{d_j^{\rm o}+2}{2}\right].$$

\begin{theorem}\label{po}Let $\TQ$ be the hyperelliptic locus of quadratic differentials in genus $g$ induced from $\TQ_0(\ldots,d^{\rm o}_i,\ldots,d^{\rm o}_j,\ldots)$. Suppose the following condition holds: 
$$\sum_{\substack{D^{\rm o}_i\not\subset B^{\rm o}\\ d^{\rm o}_i\ {\rm odd}}}1+\sum_{\substack{D^{\rm o}_j\subset B^{\rm o} \\ d^{\rm o}_j\ {\rm even}}}1 =  2.$$
Then the invariant part of the Hodge bundle over $\TQ$ splits into a direct sum of line bundles: 
$$f_*\omega_{S/\TQ} = \left(\bigoplus_{D^{\rm o}_i\not\subset B^{\rm o}}\bigoplus^{k^{\rm o}_i-1}_{j=0} f^{\rm o}_*\mathcal{O}_{D^{\rm o}_i}(\omega_{S^{\rm o}/\TQ}+D^{\rm o}-jD^{\rm o}_i)\right)\oplus\left(\bigoplus_{D^{\rm o}_j\subset B^{\rm o}}\bigoplus^{\ell^{\rm o}_j-1}_{i=0} f^{\rm o}_*\mathcal{O}_{D^{\rm o}_j}(\omega_{S^{\rm o}/\TQ}+D^{\rm o}-iD^{\rm o}_j)\right).$$
 
Moreover for a Teichm\"uller curve $C$ in $\TQ$, the direct sum splitting of $f_*\omega_{S/C}$ coincides with the direct sum of graded quotients of the Harder--Naramsihan filtration, whose Weierstrass exponents are given by $w_i^+$ as the $i$-th largest number in the set
$$\bigcup_{D^{\rm o}_i\not\subset B^{\rm o}}\left\{\frac{2l}{d^{\rm o}_i+2}\ \mid \ 0<l\leq \Big[\frac{d^{\rm o}_i+1}{2}\Big]\right\}\cup \bigcup_{D^{\rm o}_j\subset B^{\rm o}}\left\{\frac{2l-1}{d^{\rm o}_j+2}\ \mid \ 0<l\leq \Big[\frac{d^{\rm o}_j+2}{2}\Big]\right\}.$$
\end{theorem}
\begin{proof} 
Recall from \eqref{eq:o-invariant} that 
$$ f_*\omega_{S/\TQ} = f^{\rm o}_*\omega_{S^{\rm o}/\TQ}(D^{\rm o}). $$
Moreover, we have 
$$2(\omega_{S^{\rm o}/\TQ}+D^{\rm o})=2\omega_{S^{\rm o}/\TQ}+B^{\rm o}=2(f^{\rm o})^*\mathcal{L}+\sum_{D^{\rm o}_i\not\subset B^{\rm o}} d^{\rm o}_iD^{\rm o}_i+\sum_{D^{\rm o}_j\subset B^{\rm o}}(d^{\rm o}_j+1)D^{\rm o}_j.$$
Define the following divisor class 
 $$P=(\omega_{S^{\rm o}/\TQ}+D^{\rm o})-\left(\sum_{D^{\rm o}_i\not\subset B^{\rm o}} k^{\rm o}_i D^{\rm o}_i+\sum_{D^{\rm o}_j\subset B^{\rm o}}\ell^{\rm o}_j D^{\rm o}_j\right).$$
It satisfies that 
$$2P =- \sum_{\substack{D^{\rm o}_i\not\subset B^{\rm o} \\ d^{\rm o}_i\ {\rm odd}}} D^{\rm o}_i - \sum_{\substack{D^{\rm o}_j\subset B^{\rm o} \\ d^{\rm o}_j\ {\rm even}}} D^{\rm o}_j.  $$
By assumption $\deg P|_{F^{\rm o}} = -1$, 
$\deg (K_{F^{\rm o}}- P|_{F^{\rm o}})= -1$, hence $H^0(P|_{F^{\rm o}}) = 0$ and 
$H^1(P|_{F^{\rm o}})=H^0(K_{F^{\rm o}}-P|_{F^{\rm o}})^{\vee}=0$, which implies that $f^{\rm o}_*\mathcal O_{S^{\rm o}}(P)= R^1f^{\rm o}_*\mathcal O_{S^{\rm o}}(P) = 0$. 
We thus obtain that 
\begin{eqnarray*}
f^{\rm o}_*\omega_{S^{\rm o}/\TQ}(D^{\rm o}) & = & f^{\rm o}_*\omega_{S^{\rm o}/\TQ}(D^{\rm o})/ f_{*}^{\rm o} \mathcal O_{S^{\rm o}}(P) \\ 
& = & \left(\bigoplus_{D^{\rm o}_i\not\subset B^{\rm o}} f^{\rm o}_*\mathcal{O}_{k^{\rm o}_i D^{\rm o}_i}(\omega_{S^{\rm o}/\TQ}+D^{\rm o})\right)\oplus \left(\bigoplus_{D^{\rm o}_j\subset  B^{\rm o}}f^{\rm o}_*\mathcal{O}_{\ell^{\rm o}_j D^{\rm o}_j}(\omega_{S^{\rm o}/\TQ}+D^{\rm o})\right).
\end{eqnarray*}
The desired claims on the splitting and filtration thus follow from Lemma~\ref{sp1}. The values of Weierstrass exponents of $C$ can be computed by using the self-intersection number $(D^{\rm o}_j)^2$ over $C$. 
\end{proof}

We remark that in the above proof the expression by using \eqref{eq:o-invariant} looks like the situation in the proof of Proposition~\ref{hy}, but in general we cannot apply that result since here the branch divisor $B^{\rm o}$ can be different from the locus of odd singularities of the underlying quadratic differentials in $S^{\rm o}$. 

Now we apply the above two theorems to the hyperelliptic strata of quadratic differentials. 

\begin{theorem}
\label{thm:hyp-strata}
    The conditions in Theorems~\ref{hylc} and~\ref{po} hold for the hyperelliptic strata $\mathcal{Q}(4(g-k)-6,4k+2)^{\rm hyp}$, $\mathcal{Q}(2(g-k)-3,2(g-k)-3,4k+2)^{\rm hyp}$, and $\mathcal{Q}(2(g-k)-3,2(g-k)-3,2k+1,2k+1)^{\rm hyp}$. In particular, the Hodge bundle on the hyperelliptic strata splits as a direct sum of line bundles. Moreover, 
    for any Teichm\"uller curve $C$ in these hyperelliptic strata, the splitting of the Hodge bundle over $C$ coincides with the direct sum of 
    graded quotients of the Harder--Naramsimhan filtration of the Hodge bundle on $C$ whose Weierstrass exponents are given by Theorems \ref{hylc} and \ref{po}. 
\end{theorem}

\begin{proof}
Consider $\mathcal{Q}(4(g-k)-6,4k+2)^{\rm hyp}$ that arises from $\mathcal{Q}_0(2(g-k)-4,2k, -1^{2g})^{\rm hyp}$, where the hyperelliptic double cover is ramified over all the singularities. In this case the first condition in Theorem~\ref{hylc} reduces to 
$$h^0((2g-2k-3)p_1 + (2k+1)p_2) = h^0((2g-2k-4)p_1 + 2kp_2). $$
Since $p_1$ and $p_2$ are both Weierstrass points, the divisor $(2g-2k-4)p_1 + 2kp_2$ is equivalent to $K - 2p_1$ whose $h^0$ is $g-1$, and the divisor $(2g-2k-3)p_1 + (2k+1)p_2$ is equivalent to $K-p_1 + p_2$ whose $h^0$ is also $g-1$, hence the desired equality holds. The second condition in Theorem~\ref{hylc} holds by a straightforward verification. Since every singularity is ramified and exactly two of them are of even order, the condition in Theorem~\ref{po} holds as well. One can similarly check for the other two hyperelliptic strata. 
\end{proof}

Next we study the asymptotic behavior of Lyapunov exponents in the hyperelliptic case. 

\begin{theorem} 
\label{thm:hyp-asyp}
For the hyperelliptic locus of quadratic differentials induced from $\mathcal{Q}_0(\ldots,d^{\rm o}_i,\ldots,d^{\rm o}_j,\ldots)$, we have 
$$\lim_{\max\{d_k^{\rm o}\}\to \infty}\lambda_m^-=1$$
for any fixed $m$. Moreover, if there exists a fixed integer $N$ such that 
$$\sum_{\substack{D^{\rm o}_i\not\subset B^{\rm o} \\ d^{\rm o}_i\ {\rm odd}}} 1 + \sum_{\substack{D^{\rm o}_j\subset B^{\rm o} \\ d^{\rm o}_j\ {\rm even}}} 1 \leq N, $$ 
then we have 
$$\lim_{\max\{d_k^{\rm o}\}\to \infty}\lambda_m^+=1.$$
\end{theorem}
\begin{proof} Let $C$ be a Teichm\"uller curve in the hyperelliptic locus. For the first claim, recall that we have 
\begin{eqnarray*}
f_*\omega_{S/C}(D) & = & \mathcal{L}\otimes f_*\mathcal{O}_S\left(\sum_{D^{\rm o}_i\not\subset B^{\rm o}} k^{\rm o}_i(D_{i,1}+D_{i,2})+\sum_{D^{\rm o}_j\subset B^{\rm o}}(d^{\rm o}_j+1)D_j\right)\\
& \supset & 
\mathcal{L}\otimes f_*\mathcal{O}_S\left(\sum_{D^{\rm o}_i\not\subset B^{\rm o}} k^{\rm o}_i(D_{i,1}+D_{i,2})+\sum_{D^{\rm o}_j\subset B^{\rm o}}k^{\rm o}_j(2D_j)\right) \\
& = & \mathcal{L}\otimes f^{\rm o}_*\theta_*\mathcal{O}_S\left(\theta^*\Bigg(\sum_{D^{\rm o}_i\not\subset B^{\rm o}} k^{\rm o}_iD^{\rm o}_i+\sum_{D^{\rm o}_j\subset B^{\rm o}}k^{\rm o}_jD^{\rm o}_j\Bigg)\right) \\
& = & \mathcal{L}\otimes f^{\rm o}_*\left(\big(\mathcal{O}_{S^{\rm o}}\oplus \mathcal{O}_{S^{\rm o}}(-D^{\rm o})\big)\otimes\mathcal{O}_{S^{\rm o}}\Bigg(\sum k^{\rm o}_i D^{\rm o}_i\Bigg)\right) \\
&\supset & \mathcal{L}\otimes f^{\rm o}_*\mathcal{O}_{S^{\rm o}}\left(k^{\rm o}_l D^{\rm o}_l\right) = \bigoplus^{k^{\rm o}_l}_{j=0} \big(\mathcal{L}\otimes f^{\rm o}_*\mathcal{O}_{D^{\rm o}_l}(jD^{\rm o}_l)\big)
\end{eqnarray*}
where $k_l^{\rm o} = \max\{ k_i^{\rm o} \}$.  As $\max\{d_k^{\rm o} \}\to \infty$, we have $k_l^{\rm o} = [(d_l^{\rm o}+1)/2]\to \infty$, and hence the Weierstrass exponent $w_m^-$ of the above direct sum subbundle approaches $1$ for any fixed $m$. 
As in the proof of Corollary~\ref{as}, we conclude that 
$$\lim_{\max\{d_k^{\rm o}\}\to\infty}\lambda_m^- = 1.$$

For the other claim, define the following divisor class $$Q=(\omega_{S^{\rm o}/C}+D^{\rm o})-\left(\sum_{D^{\rm o}_i\not\subset B^{\rm o}} (\ell^{\rm o}_i-1)D^{\rm o}_i+\sum_{D^{\rm o}_j\subset B^{\rm o}} k^{\rm o}_j D^{\rm o}_j\right).$$ It satisfies that $$2Q =\sum_{\substack{D^{\rm o}_i\not\subset B^{\rm o} \\ d^{\rm o}_i\ {\rm odd}}} D^{\rm o}_i + \sum_{\substack{D^{\rm o}_j\subset B^{\rm o} \\ d^{\rm o}_j\ {\rm even}}} D^{\rm o}_j.  $$ 
Since $\deg Q|_{F^{\rm o}} \geq 0$, we have 
$R^1f^{\rm o}_*\mathcal O_{S^{\rm o}}(Q) = 0$. 
Consider the exact sequence 
$$0\rightarrow V_1=f^{\rm o}_*\mathcal{O}_{S^{\rm o}}(Q)\rightarrow  V=f^{\rm o}_*\omega_{S^{\rm o}/C}(D^{\rm o})\rightarrow V_2=f^{\rm o}_*\omega_{S^{\rm o}/C}(D^{\rm o})/ f_{*}^{\rm o} \mathcal O_{S^{\rm o}}(Q)\rightarrow 0.$$
It follows that 
$$V / V_1 = V_2 = \left(\bigoplus_{D^{\rm o}_i\not\subset B^{\rm o}} f^{\rm o}_*\mathcal{O}_{(\ell^{\rm o}_i-1)D^{\rm o}_i}(\omega_{S^{\rm o}/C}+D^{\rm o})\right)\oplus \left(\bigoplus_{D^{\rm o}_j\subset  B^{\rm o}}f^{\rm o}_*\mathcal{O}_{k^{\rm o}_j D^{\rm o}_j}(\omega_{S^{\rm o}/C}+D^{\rm o})\right).$$
By the Shatz specialization, we know that the Harder--Narasimhan polygon of $V$ lies below that of $V_1\oplus V_2$. Note that 
$${\rm rank}\ V_1=\frac{1}{2}\left(\sum_{\substack{D^{\rm o}_i\not\subset B^{\rm o} \\ d^{\rm o}_i\ {\rm odd}}} 1 + \sum_{\substack{D^{\rm o}_j\subset B^{\rm o} \\ d^{\rm o}_j\ {\rm even}}} 1\right) + 1 $$
is bounded by assumption. 
Therefore, as $\max\{d_k^{\rm o}\}\to\infty$, the  Harder--Narasimhan polygon of $V_2$ can be  arbitrarily close to that of $V$. By using 
$V_2$ we thus conclude that 
$$\lim_{\max\{d_k^{\rm o}\}\to \infty}\lambda_m^+ = 1.$$ 
\end{proof}

\begin{corollary} For each of the hyperelliptic strata $\mathcal{Q}(2(g-k)-3,2(g-k)-3,2k+1,2k+1)^{\rm hyp}$, $\mathcal{Q}(2(g-k)-3,2(g-k)-3,4k+2)^{\rm hyp}$, and $\mathcal{Q}(4(g-k)-6,4k+2)^{\rm hyp}$, we have
$$\lim_{k\to\infty}\lambda_m^+=\lim_{g\to\infty}\lambda_m^+=\lim_{k\to\infty}\lambda_m^-=\lim_{g\to\infty}\lambda_m^-=1.$$
\end{corollary}
\begin{proof}
Since the hyperelliptic strata satisfy the assumption of Theorem~\ref{thm:hyp-asyp}, the claim thus follows from there. 
\end{proof}

Besides pulling back quadratic differentials from strata in genus zero, we can also pull back from strata in other genera by using the same construction as in the sequence \eqref{eq:triple}. In what follows we demonstrate this idea by pulling back from the nonvarying strata in genus one. 

\begin{proposition} 
For the loci of quadratic differentials induced from the double cover construction by pulling back quadratic differentials from $\mathcal{Q}_1(n,-1^{n})$ and $\mathcal{Q}_1(n,1,-1^{n+1})$, for any fixed $m$ we have
$$\lim_{n\to\infty}\lambda_m^-=1.$$   
\end{proposition}
\begin{proof} The proof is similar to that of Theorem~\ref{thm:hyp-asyp}. We have 
\begin{eqnarray*}
f_*\omega_{S/C}(D) & = & \mathcal{L}\otimes f_*\mathcal{O}_S\left(\sum_{D^{\rm o}_i\not\subset B^{\rm o}} k^{\rm o}_i (D_{i,1}+D_{i,2})+\sum_{D^{\rm o}_j\subset B^{\rm o}}(d^{\rm o}_j+1)D_j\right)\\
& \supset & \mathcal{L}\otimes f^{\rm o}_*\theta_*\mathcal{O}_S\left(\theta^*\left(\Big[\frac{n+1}{2}\Big]D^{\rm o}_1\right)\right) \\
& \supset & \mathcal{L}\otimes f^{\rm o}_*\mathcal{O}_{S^{\rm o}}\left(\left[\frac{n+1}{2}\right]D^{\rm o}_1\right) =  \bigoplus^{[\frac{n+1}{2}]}_{i=0}\big(\mathcal{L}\otimes f^{\rm o}_*\mathcal{O}_{D^{\rm o}_1}(iD^{\rm o}_1)\big). 
\end{eqnarray*}
By using the above direct sum subbundle we conclude that 
$$\lim_{n\to\infty}\lambda_m^-=1.$$
\end{proof}

\section{Upper bounds of the Harder--Narasimhan polygon}
\label{sec:bounds}
An analogous discussion, similar to the one in \cite{YZ2}, can provide an upper bound for $w_i^+$ and $w_i^-$ for all Teichm\"uller curves in a stratum of quadratic differentials. 

For a vector bundle $V$, recall that   $$\mu_i(V)=\mu({\rm gr}_j^{{\rm HN}(V)})$$ 
if ${\rm rank}\ {\rm HN}_{j-1}(V)<i\leq {\rm rank}\ {\rm HN}_j(V)$, where ${\rm gr}_j^{{\rm HN}(V)}$ is the $j$-th graded quotient of the Harder--Narasimhan filtration of $V$ and $\mu$ is the slope as we reviewed in Section~\ref{sec:filtration}. 

\begin{lemma}[{\cite[Lemma 2.1]{YZ2}}]\label{ie}  Let $V$ and $U$ be two vector bundles of rank $n$ on a compact algebraic curve $C$, with increasing filtrations 
    $$0\subset V_0\subset V_1\subset \cdots\subset V_n=V,$$
    $$0\subset U_0\subset U_1\subset \cdots\subset U_n=U,$$
   such that $V_i/V_{i-1}$ and $U_i/U_{i-1}$ are line bundles, $V_i/V_{i-1}\subset U_i/U_{i-1}$, and the degrees of $U_i/U_{i-1}$ decrease as $i$ increases. Then for any $1\leq i\leq n$, we have $\mu_i(V)\leq \deg (U_i/U_{i-1})$.
\end{lemma}

\begin{theorem}For any Teichm\"uller curve $C$ in $\mathcal{Q}_g(\ldots,d_j,\ldots)$, we have
$w_i^+$ for $i=1,\ldots,g$ less than or equal to the $N_i$-th (defined in the proof) smallest number in the set 
    $$\left\{\frac{2k}{d_j+2}\ |\ 0< 2k\leq d_j+1,\forall j \right\}.$$
 Moreover, $N_i\leq 2g-2i$ for $1\leq i<g$.
\end{theorem}

\begin{proof}
Recall the divisor classes $B = 2D$ introduced in Section~\ref{sec:quad},  where $B\subset S$ is the branch locus. If $B$ is nonempty, then the relative degree of $D$ on each fiber is positive. If $B = 0$, since the canonical double cover is connected, $D$ is a nontrivial $2$-torsion. We thus have $f_{*}\mathcal O_S(-D) = 0$ in each case. It follows that 
$$f_*\omega_{S/C}\left(-\sum k_jD_j\right)= f_*\left(f^*\mathcal{L}\otimes\mathcal{O}_S(-D)\right)=0, $$ 
which implies that 
   \begin{eqnarray*}
   f_*\omega_{S/C} &\subset & f_*\mathcal{O}_{\sum k_jD_j}(\omega_{S/C}) \\
   & = & \bigoplus f_*\mathcal{O}_{k_jD_j}(\omega_{S/C}) \\
   & = & \bigoplus_j\bigoplus^{k_j-1}_{i=0} f_*\mathcal{O}_{D_j}(\omega_{S/C}-iD_j).
   \end{eqnarray*}

Let $M_l$ be the line bundle with the $l$-th smallest degree (counting with multiplicities) in the set
$$\left\{f_*\mathcal{O}_{D_j}(\omega_{S/C}-iD_j)\ \mid\ 0\leq i\leq k_j-1,\forall j \right\}.$$
We rearrange the following $\sum k_j$ sections (counting with multiplicities) 
$$\ldots,\underbrace{D_j,\ldots,D_j}_{k_j},\ldots$$
to be 
$$D'_1,D'_2,\ldots,D'_{\sum k_j}$$
such that 
$$G_l={\rm grad}\left({\rm HN}
\left(f_*\mathcal{O}_{\sum k_jD_j-\sum^{l}_{i=1} D'_i}\left(\omega_{S/C}-\sum^{l}_{i=1} D'_i\right)\right)\right)$$ 
has rank equal to $\sum k_j-l$. Then we have 
$$0\subset \cdots \subset  G_{l+1}\subset G_l\subset\cdots\subset \bigoplus_j\bigoplus^{k_j-1}_{i=0}f_*\mathcal{O}_{D_j}(\omega_{S/C}-iD_j)$$
where $G_l/G_{l+1}=M_l$ and  $\deg (G_l/G_{l+1})$ increases in $l$.

Denote by 
$$W_l= f_*\mathcal{O}_S\left(\omega_{S/C}-\sum^{l}_{i=1} D'_i\right).$$  Then we have 
$${\rm rank}\ W_l=h^0\left(\left(\omega_{S/C}-\sum^{l}_{i=1} D'_i\right)\Big|_F\right)=h^0\left(\sum^{l}_{i=1} D'_i|_F\right)-l+g-1, $$ 
where $F$ is a generic fiber. For each $j$ such that rank $W_{l+1} = j-1$ and rank $W_l =j$, define $N_j = l$. We then denote $V_j:=W_{N_j}=\cdots =W_{N_{j+1}+1}$. This gives a filtration of $f_*\omega_{S/C}$ as follows:  
$$0\subset \cdots \subset V_{j-1}\subset V_j \subset \cdots \subset f_*\omega_{S/C}.$$
Consider another filtration
$$0\subset \cdots \subset U_{k-1}=\bigoplus^{k-1}_{j=1} M_{N_j}  \subset U_k=\bigoplus^k_{j=1} M_{N_j} \subset \cdots \subset U_{g}.$$
We know that 
$$V_j/V_{j-1}=W_{N_j}/W_{N_j+1}\hookrightarrow  f_*\mathcal{O}_{D'_{N_j+1}}\left(\omega_{S/C}-\sum^{N_j}_{i=1} D'_i\right)=M_{N_j}=U_j/U_{j-1}.$$
Therefore, the two filtrations satisfy the assumption of Lemma \ref{ie}. It follows that 
$$w_i^+=\mu_{i}(f_*\omega_{S/C})/\deg \mathcal{L}$$ is less than or equal to the $N_i$-th smallest number in the set $$\left\{\frac{2k}{d_j+2}\ \mid \ 0<2k\leq d_j+1,\forall j \right\}.$$
Finally by Clifford's theorem (as used in \cite[Corollary 5.3]{YZ2}), $N_i\leq 2g-2i$ for $1\leq i<g$. 
\end{proof}

\begin{theorem}For any Teichm\"uller curve $C$ in $\mathcal{Q}_g(\ldots,d_j,\ldots)$, we have
$w_1=w_1^-=1$ and $w_i^-$ for $i=2,\ldots,g_{\rm eff}$ less than or equal to the $H_i$-th (defined in the proof) largest number in the set 
    $$\left\{1-\frac{2k}{d_j+2}\ \mid\ 0< 2k\leq d_j+1,\forall j \right\}.$$
    Moreover, $H_i\geq 2i-2$ for $1<i<g$ and $H_i=i+g-1$ for $i\geq g$. 
\end{theorem}
\begin{proof}
The idea here is similar to the preceding proof. We have
\begin{eqnarray*}
f_*\omega_{S/C}(D)/\mathcal{L} &\subset & \mathcal{L}\otimes f_*\mathcal{O}_{\sum k_jD_j}\left(\sum k_jD_j\right) \\
& = & \bigoplus \big(\mathcal{L}\otimes f_*\mathcal{O}_{k_jD_j}(k_jD_j)\big) \\
  &=&\bigoplus_j\bigoplus^{k_j}_{i=1}\big(\mathcal{L}\otimes f_*\mathcal{O}_{D_j}(iD_j)\big).
   \end{eqnarray*}

Let $M_l$ be the line bundle with the $l$-th largest degree (counting with multiplicities) in the set
$$\left\{\mathcal{L}\otimes f_*\mathcal{O}_{D_j}(iD_j)\ \mid \ 1\leq i\leq k_j,\forall j \right\}.$$
We rearrange the $\sum k_j$ sections (counting with multiplicities) 
$$\ldots,\underbrace{D_j,\ldots,D_j}_{k_j},\ldots$$
to be 
$$D'_1,D'_2,\cdots,D'_{\sum k_j}$$
such that for 
$$G_l={\rm grad}\left({\rm HN}\left(\mathcal{L}\otimes f_*\mathcal{O}_{\sum^{l}_{i=1} D'_i}\left(\sum^{l}_{i=1} D'_i\right)\right)\right)$$ 
we have 
$$0\subset \cdots \subset G_{l-1} \subset  G_l\subset\cdots\subset \bigoplus\bigoplus^{k_j}_{i=1}\big(\mathcal{L}\otimes f_*\mathcal{O}_{D_j}(iD_j)\big)$$
where $G_l/G_{l-1}=M_l$ and  ${\rm deg}(G_i/G_{i-1})$ decreases in $l$.

Denote by 
$$W_l=\mathcal{L}\otimes f_*\mathcal{O}_S\left(\sum^{l}_{i=1} D'_i\right)/\mathcal{L}.$$
Then we have 
$${\rm rank}\ W_l=h^0\left(\sum^{l}_{i=1} D'_i|_F\right)-1, $$ 
where $F$ is a general fiber. For each $j$ such that rank $W_{l-1} = j-1$ and rank $W_l =j$, define $H_j = l$. We then denote $V_j:=W_{H_j}=\cdots =W_{H_{j+1}-1}$. This gives a filtration of $f_*\omega_{S/C}(D)/\mathcal{L}$ as follows: 
$$0\subset \cdots \subset V_{j-1}\subset V_j \subset \cdots \subset f_*\omega_{S/C}(D)/\mathcal{L}.$$
Consider another filtration
$$0\subset \cdots \subset U_{k-1}=\bigoplus^{k-1}_{j=1} G_{H_j}/G_{H_j-1}  \subset U_k=\bigoplus^k_{j=1} G_{H_j}/G_{H_j-1} \subset \cdots \subset U_{g_{\rm eff}}.$$
We know that 
$$V_j/V_{j-1}=W_{H_j}/W_{H_j-1}\hookrightarrow \mathcal{L}\otimes f_*\mathcal{O}_{D'_{H_j}}\left(\sum^{H_j}_{i=1} D'_i\right)=G_{H_j}/G_{H_j-1}=U_j/U_{j-1}.$$
Therefore, the two filtrations satisfy the assumption of Lemma \ref{ie}. It follows that 
$$w_i^-=\frac{\mu_{i-1}(f_{*}\omega_{S/C}(D)/\mathcal L)}{\deg \mathcal{L}}$$ is less than or equal to the $H_i$-th largest number in the set $$\left\{1-\frac{2k}{d_j+2}\ \mid\ 0< 2k\leq d_j+1,\forall j \right\}.$$
Finally by Clifford's theorem, $H_i\geq 2i-2$ for $1<i<g$, and by Riemann--Roch, $H_i=i+g-1$ for $i\geq g$.
\end{proof}

Next we will establish an upper bound of Weierstrass exponents for the Hodge bundle on general Teichm\"uller curves. This bound is closely related to the description of the Harder--Narasimhan polygons for the non-varying strata. 

\begin{proposition}\label{up}
For a general Teichm\"uller curve $C$ in the stratum $\mathcal{Q}_g(d_1,\ldots,d_n,-1^l)$ with $d_i>0$ and $l\geq g$, the Weierstrass exponents of the Hodge bundle on $C$ satisfy that 
\begin{itemize}
    \item  $w_i^+$ for $i=1,\ldots,g$ is less than or equal to the $(g-i+1)$-st smallest number in the set 
    $$\left\{\frac{2k}{d_j+2}\ \mid\ 0< 2k\leq d_j+1,\ j=1,\ldots,n\right\}.$$  
    \item $w_1^-=1$ and $w_i^-$ for $i=2,\ldots,g_{\rm eff}$ is less than or equal to the $(g_{\rm eff}-i+1)$-st smallest number in the set 
    $$\left\{1-\frac{2k}{d_j+2}\ \mid\ 0<2k\leq d_j+1,\ j=1,\ldots,n\right\}.$$
\end{itemize}    
\end{proposition}
\begin{proof}
First note that 
$$\dim \mathcal{Q}_g(d_1,\ldots,d_n,-1^l) / \mathbb C^{*} = 2g-3+n+l \geq 3g-3+n = \dim \mathcal M_{g,n}.$$ 
By \cite{bud} we know that for a general smooth curve $X$ of genus $g$ with general points $p_1, \ldots, p_n\in X$, there exists $(X, q)\in \mathcal{Q}_g(d_1,\ldots,d_n,-1^l)$ such that 
$q$ has zeros at $p_i$ with order $d_i$. 
Next choose $0\leq a_i\leq k_i = [(d_i+1)/2]$ such that $\sum a_i=g$. We claim that $h^0(X,\sum a_ip_i) = 1$ where $p_1, \ldots, p_n\in X$ are general. Otherwise if $h^0(X,\sum a_ip_i)\geq 2$, then by semi-continuity the same holds for all $p_1, \ldots, p_n \in X$, including the special case where $p_1 = \cdots = p_n = p$ for an arbitrary point $p\in X$. However, this would imply that $h^0(X,gp) \geq 2$, which is impossible if $p$ is not a Weierstrass point. Therefore, we conclude that $h^0(\sum a_i p_i)=1$ and $h^0(K-\sum a_ip_i)=0$ on $X$, which implies that $f_* \big(\omega_{S/C}(-\sum a_iD_i)\big)=0$ and $f_* \mathcal O_S(\sum a_i D_i)=\mathcal{O}_C$ since these direct images are vector bundles with generic fiber rank zero and one respectively. We thus obtain the following inclusions:
\begin{eqnarray*}    
f_* \omega_{S/C} & = & f_* \omega_{S/C} /f_* \Big(\omega_{S/C}\Big(-\sum a_iD_i\Big)\Big)\\
& \subset & f_*\mathcal{O}_{\sum a_iD_i}(\omega_{S/C}) \\
& = & \bigoplus_i\bigoplus^{a_i-1}_{j=0}f_*\mathcal{O}_{D_i}(\omega_{S/C}-jD_i)
\end{eqnarray*} 
and 
\begin{eqnarray*}
f_* \omega_{S/C}(D)/\mathcal{L} & = & \mathcal{L}\otimes \left(f_*\mathcal O_S\Big(\sum k_iD_i\Big)/f_* \mathcal O_S\Big(\sum a_i D_i\Big)\right)\\
&\subset & 
\mathcal{L}\otimes f_*\mathcal{O}_{\sum (k_i-a_i)D_i}\left(\sum k_iD_i\right)\\
& = & \bigoplus_i \bigoplus^{k_i}_{j=a_i+1}\big(\mathcal{L}\otimes f_*\mathcal{O}_{D_i}(jD_i)\big). 
\end{eqnarray*}
  Now the desired claim follows from Lemma~\ref{ie} since the filtrations of $f_* \omega_{S/C}$ and $f_*\mathcal{O}_{\sum a_iD_i}(\omega_{S/C})$ satisfy its assumption and the filtrations of $f_*\omega_{S/C}(D)$ and $\mathcal{L}\otimes f_*\mathcal{O}_{\sum (k_i-a_i)D_i}(\sum k_iD_i)$ also satisfy the assumption as shown in the proofs of the preceding two theorems. 
\end{proof}

A Teichm\"uller curve with a Forni-subspace of rank $2d$ has a fixed part of dimension $d$ (see \cite{Au} for the definition and background). Our study implies the following result about the fixed part of the Hodge bundle on Teichm\"uller curves in strata of quadratic differentials (analogous to the case of strata of Abelian differentials in  \cite[Proposition 4.5]{BHM}).

\begin{corollary}\label{no}
All Lyapunov exponents of any Teichm\"uller curve $C$ in the non-varying strata of quadratic differentials are nonzero. 
\end{corollary}
\begin{proof}
The variation of Hodge structures over $C$ decomposes as $$R^1f'_*\mathbb{C}\cong \mathbb{L}\oplus \mathbb{U}\oplus \mathbb{M},$$ 
where $\mathbb{L}$ is the maximal Higgs bundle, $\mathbb{U}$ is the unitary summand stemming from the fixed part, and $\mathbb{M}$ is the remaining part (see e.g., \cite[Section 2]{Mo}). In particular, 
$$\mathrm{grad}({\rm HN}(f'_*\omega_{S'/C}))\cong \mathrm{grad}({\rm HN}(\mathbb{L}^{(1,0)}))\oplus\mathrm{grad}({\rm HN}(\mathbb{U}^{(1,0)}))\oplus\mathrm{grad}({\rm HN}(\mathbb{M}^{(1,0)})).$$
Since all $w_i^+$ and $w_i^-$ are shown to be nonzero for $C$ in the non-varying strata, there does not exist a $\mathrm{grad}({\rm HN}(\mathbb{U}^{(1,0)}))$-term, and hence there does not exist any unitary part $\mathbb{U}$ in either the invariant part or the anti-invariant part of the Hodge bundle. The desired result thus follows because zero Lyapunov exponents come from the unitary part.
\end{proof}

It would be interesting to determine whether there can be any new non-varying strata besides those in \cite{CM2}. Using our results as evidences, we conclude the paper by the following speculation. 

\begin{conjecture}
Suppose $\mathcal{Q}_g(d_1,\ldots,d_n)$ is neither hyperelliptic nor irregular. Then it is a 
non-varying stratum if and only if the value of $L^+$ of the stratum is equal to the sum of the $g$ smallest numbers in the set
$$\left\{\frac{2k}{d_j+2}\ |\ 0<2k\leq d_j+1,\ j=1,\ldots,n\right\}.$$
\end{conjecture}

\begin{example} For the stratum $\mathcal{Q}(1^n,1^{-n})$, by  \cite[Corollary 1.8]{CMS} we know that 
$$L^+=w_1^+=\frac{2}{1+\frac{(2n-2)!!}{(2n-3)!!}}. $$ 
But the smallest value of $2k / (d_j+2)$ in the above set is $2/3 \neq L^+$ for $n\geq 3$. 
Indeed by numerical evidences (e.g., by using Teichm\"uller curves generated by square-tiled surfaces) we expect $\mathcal{Q}(1^n,1^{-n})$ to be a varying stratum for $n\geq 3$. 
\end{example}

\noindent Department of Mathematics, Boston College, Chestnut Hill, MA 02467, USA\\
Email address: dawei.chen@bc.edu\\

\noindent School of Mathematical Sciences, Zhejiang University, Hangzhou, China\\ Email address: yufei@zju.edu.cn

\end{document}